\documentclass[leqno,final]{article}

\usepackage{setspace} 
\usepackage{geometry}
\usepackage{fancyhdr}
\usepackage{cite}
\usepackage{graphicx}
\usepackage{epstopdf}
\usepackage{subfig}
\usepackage{tikz}
\usetikzlibrary{positioning, shapes.geometric}
\usetikzlibrary{decorations.markings}
\usetikzlibrary{arrows.meta}

\usepackage{color}
\usepackage{xcolor}
\usepackage{hyperref}

\hypersetup{hidelinks}
\usepackage{indentfirst}
\usepackage{url}
\usepackage{amsfonts,amsmath,amssymb,amsthm,latexsym,lineno}
\usepackage{enumerate}
\usepackage{bm}
\usepackage{algorithm,algorithmic}
\usepackage[normalem]{ulem}
\usepackage{cleveref}

\ifpdf
 \DeclareGraphicsExtensions{.eps,.pdf,.png,.jpg}
\else
 \DeclareGraphicsExtensions{.eps}
\fi

\newtheorem{theorem}{Theorem}[section]
\newtheorem{lemma}[theorem]{Lemma}
\newtheorem{remark}[theorem]{Remark}
\newtheorem{example}[theorem]{Example}

\numberwithin{theorem}{section}
\numberwithin{equation}{section}
\numberwithin{figure}{section}
\numberwithin{algorithm}{section}



\newcommand{\be}{\begin{equation}}
\newcommand{\ee}{\end{equation}}
\newcommand{\ba}{\begin{array}}
\newcommand{\ea}{\end{array}}
\newcommand{\ben}{\begin{eqnarray}}
\newcommand{\een}{\end{eqnarray}}
\newcommand{\bn}{\begin{eqnarray*}}
\newcommand{\en}{\end{eqnarray*}}
\newcommand{\p}{\partial}

\newcommand{\bv}{{\bf v}}
\newcommand{\bbn}{{\bf n}}

\newcommand{\cQ}{{\cal Q}}

\newcommand{\cE}{{\cal E}}
\newcommand{\cMh}{{\cal M}_h}
\newcommand{\cTh}{{\cal T}_h}

\newcommand{\cT}{\mathcal{T}}

\newcommand{\cTm}{{\cal T}_{h,i}}
\newcommand{\cMm}{{\cal M}_{h,i}}

\newcommand{\cM}{\mathcal{M}}

\newcommand{\al}{\alpha}
\newcommand{\de}{\delta}

\newcommand{\ga}{\gamma}
\newcommand{\Ga}{\Gamma}
\newcommand{\ka}{\kappa}

\newcommand{\Om}{\Omega}

\newcommand{\rd}{{\,\rm d}}
\newcommand{\norm}[1]{\left\Vert#1\right\Vert}
\newcommand{\norme}[1]{\left\Vert{\hskip -2.6pt}\left\vert #1 \right\vert{\hskip -2.6pt}\right\Vert_{h}}
\newcommand{\normee}[1]{\big\vert{\hskip -1.5pt}\big\vert{\hskip -1.5pt}\big\vert #1 \big\vert{\hskip -1.5pt}\big\vert{\hskip -1.5pt}\big\vert_{h}}
\newcommand{\abs}[1]{\left\vert#1\right\vert}

\newcommand{\set}[1]{\left\{#1\right\}}
\newcommand{\av}[1]{\left\{#1\right\}}
\newcommand{\jm}[1]{\left[#1\right]}
\newcommand{\eq}[1]{\begin{align}#1\end{align}}
\newcommand{\eqn}[1]{\begin{align*}#1\end{align*}}

\newcommand{\oo}{\Om_1\cup\Om_2}
\newcommand{\nn}{\nonumber}

\DeclareSymbolFont{ugmL}{OMX}{mdugm}{m}{n}
\SetSymbolFont{ugmL}{bold}{OMX}{mdugm}{b}{n}
\DeclareMathAccent{\wideparen}{\mathord}{ugmL}{"F3}

\title{An Unfitted Interface Penalty DG--FE Method for Elliptic Interface Problems}

\author{
Juan Han\thanks{Department of Mathematics, Nanjing University, Nanjing 210093, China.
Email address: dg20210004@smail.nju.edu.cn}
\,\,\,
Haijun Wu\thanks{Department of Mathematics, Nanjing University, Nanjing 210093, China. 
Email address: hjw@nju.edu.cn}
\,\,\,\,\mbox{and}\,\,\,
Yuanming Xiao\thanks{Department of Mathematics, Nanjing University, Nanjing 210093, China.
Email address: xym@nju.edu.cn}
}

\date{}

\begin{document}
	\maketitle

	\begin{abstract}
		We propose an unfitted interface penalty Discontinuous Galerkin-Finite Element Method (UIPDG-FEM) for elliptic interface problems. This hybrid method combines the interior penalty discontinuous Galerkin (IPDG) terms near the interface-enforcing jump conditions via Nitsche method-with standard finite elements away from the interface. The UIPDG-FEM retains the flexibilities of IPDG, particularly simplifying mesh generation around complex interfaces, while avoiding its drawback of excessive number of global degrees of freedom. We derive optimal convergence rates independent of interface location and establish uniform flux error estimates robust to discontinuous coefficients. To deal with conditioning issues caused by small cut elements, we develop a robust two-dimensional merging algorithm that eliminates such elements entirely, ensuring the condition number of the discretized system remains independent of interface position. A key feature of the algorithm is a novel quantification criterion linking the threshold for small cuts to the product of the maximum interface curvature and the local mesh size.  Numerical experiments confirm the theoretical results and demonstrate the effectiveness of the proposed method.
	\end{abstract}
	

	\section{Introduction}
	Interface problems have a wide range of applications in engineering and science, particularly in fluid dynamics and materials science.
	Examples include multi-phase flows, fluid-structure interaction and applications with varying domains, such as shape or topology optimisation.
    In this paper, we consider the following model problem: Let $\Omega=\Omega_1\cup\Gamma\cup\Omega_2$ be a bounded and convex polygonal/polyhedral domain in $\mathbb{R}^d~(d=2$ or $3)$, where $\Omega_1$ and $\Omega_2$ are two subdomains of $\Omega$ separated by a $C^2$-smooth interior interface $\Gamma$
\setlength\abovedisplayskip{3pt}
\setlength\belowdisplayskip{3pt}
\be\label{elliptic-interface}
	\begin{cases}
		\ba {ll}
		-\nabla \cdot (\al\nabla u) =  f \quad & {\rm in} \quad \Omega_1\cup\Omega_2, \\
		\jm{u}= g_D, \quad \big[\al\nabla u\cdot\bbn\big]= g_N \quad & {\rm on} \quad \Gamma, \\
		u=0 \quad & {\rm on} \quad \p\Omega.
		\ea
	\end{cases}
\ee
Here $\alpha=\alpha_i,i=1,2,$ is a piecewise constant function on the partition, and the jump across $\Gamma$ is $[v]=(v_1-v_2)|_\Gamma$, where $v_i=v|_{\Omega_i}, i=1,2$. We assume $\partial\Omega \cap \partial \Omega_1=\emptyset$, and $\bbn$ is unit outward normal to the boundary of $\Omega_1$. Due to the discontinuity of the coefficient $\alpha$ and the irregular geometry of the interface $\Gamma$, naive variational formulations may exhibit poor stability and loss in accuracy across $\Gamma$.

One way to remedy the accuracy of approximation is to use body-fitted grids. These methods align mesh vertices with the interface to resolve discontinuities directly, enabling application of classical finite element schemes. For linear elements, optimal convergence is achieved with interface-aligned meshes \cite{CZ98}, while higher-order accuracy is attained via isoparametric elements \cite{lmwz10}. Further, the strict shape-regularity requirement for body-fitted grids can be relaxed to maximal-angle-bounded meshes while preserving optimal approximation properties in linear element spaces \cite{chen2015adaptive,hu2021optimal}. 
	
To reduce the high cost of mesh generation, there has been a notable surge in interest towards {\it immersed} or {\it unfitted} methods. These approaches aim to simulate problems using a simple background mesh while allowing discontinuities to be distributed within elements. Within the finite element framework, interface conditions on interface elements are typically handled in one of two ways: (1) constructing special element bases that satisfy the conditions approximately or weakly (e.g., as in the immersed finite element method \cite{LLW03,LLZ15}), or (2) directly employing sliced polynomials on interface elements. This work focuses on the latter approach, where two independent sets of degrees of freedom are assigned to interface elements (or, ``doubling of unknowns'' in \cite{HH02}). Dirichlet jumps across interfaces are weakly enforced via penalty or Nitsche formulations \cite{Nitsche,Babuvska72}. Numerous numerical methods in this category have been developed and analyzed, including the multiscale unfitted finite element method \cite{cgh10}, the cut finite element method \cite{bhl17}, the $hp$-interface penalty finite element method \cite{WX10}, and unfitted discontinuous Galerkin methods \cite{m09, bastian2009unfitted}. For other {\it immersed} or {\it unfitted} techniques, we refer to the immersed boundary method \cite{Pe77}, the immersed interface method \cite{ll94}, the ghost fluid method \cite{FAMO99}, and the kernel-free boundary integral method \cite{zy2024}, among others cited therein.

Despite their utility, conditioning issues remain a primary obstacles to the successful implementation of unfitted finite element methods in realistic large-scale applications. This issue is associated with the so-called small cut cell problem: an element (say $K$) may have arbitrarily small intersection with one sub-domain, that is, $|K\cap\Omega_i|/|K|<\delta\ll1, i=1,2$. For brevity, we refer to them as small elements. Without specific stabilization, the appearance of such elements can result in significantly ill-conditioned discrete algebraic systems \cite{SFA18}. Two primary strategies exist to address the adverse effects of small elements on the conditioning of discrete systems. The first approach introduces additional penalties (termed ``ghost penalties'') on the jumps of derivatives across faces of interface elements \cite{Bu10}. The resulting formulation is weakly non-consistent but robust to cut location. The second strategy couples the degrees of freedom (DOFs) of small elements with those of adjacent larger elements to mitigate conditioning degradation. For instance, the patch reconstruction DG method \cite{LY20} integrates discontinuous Galerkin (DG) formulations with least-squares stabilization, while the aggregated unfitted finite element method \cite{SFA18} relies on a stable extrapolation extension operator from well-posed to ill-posed DOFs.
Meanwhile, the merging element methods \cite{JL13,hwx17,chen2022} combine the small elements with their neighboring large elements to form macro-elements, ensuring the intersection of each macro-element and $\Omega_i$ retains sufficient size. This idea extends to unfitted hybrid high-order (HHO) method with cell agglomeration \cite{BCDE21} which allows meshes including polygonal elements. Merging strategies have been adapted for non-conforming meshes \cite{chen2022} and unfitted HHO discretizations \cite{BCDE21}, demonstrating broad applicability across mesh types.

In this work, we present a novel algorithm (\cref{mergingAlgorithm}) for robustly merging small interface elements with adjacent larger elements to construct stable macro-elements. Under the interface resolution and merging criteria in Assumptions \eqref{resolvingInterface}-\eqref{merging}, we rigorously prove (1) the feasibility of the algorithm: ensuring every small element has a mergeable neighbor, and (2) the reliability of the algorithm: guaranteeing non-overlapping macro-elements with local mesh sizes equivalent to the original background mesh. The success of the algorithm relies principally on Assumption \eqref{merging} in Section \ref{mergingSection}: Letting $\ka_m=\max_{P\in\Ga} |\ka(P)|$ denote the maximum curvature of the interface $\Ga$, if the background mesh resolves the interface such that $t=\ka_m h\le1$ and
\[\frac{30t(t+2)}{100t+63}\le 1-2\de,\]
then every small element $K$ can be merged with at most three adjacent elements to form a rectangular macro-element $M_i(K)$ that retains sufficient size within $\Omega_i$ for $i=1,2$.
This criterion defines the relationship between the maximum curvature of the interface $\ka_m$, the mesh size $h$ and the small-element threshold $\de$.
For $\de=0.25$, the criterion simplifies to $\ka_m h\le 0.87$. When $\de\le 0.22$, it further reduces to $\ka_m h\le 1$, permitting the turning angle of
$\Ga$ within any single element to be as large as $\frac{\pi}{2}$.
The approach facilitates the merging process near geometrically complex interfaces and removes the restrictive assumption (common in prior work \cite{chen2022,BCDE21}) that an element may intersect the interface at most twice (see Assumptions \eqref{resolvingInterface}).

Based on the merged meshes, we propose an {\it unfitted interface penalty DG-FE method} (UIPDG-FEM) for the problem \cref{elliptic-interface}. This hybrid approach integrates two established numerical strategies:
The IPDG method is applied locally within macro-element regions, while the standard FEM is used elsewhere. Penalty terms are imposed along the interface $\Gamma$ to enforce the jump conditions and across DG-FEM coupling interfaces to ensure consistent integration of both discretization schemes.
The UIPDG-FEM inherits the flexibilities of IPDG method while avoiding its global DOF inflation as DG scheme is confined to a small portion of elements near the interface.
Furthermore, robustness with respect to coefficient contrasts is ensured by incorporating harmonic weighting into the Nitsche formulation, following the approach in \cite{Bu10,hwx17}. In particular, the UIPDG-FEM does not suffer the problem of instability since the extrapolation is not used here. The UIPDG-FEM exhibits desirable characteristics possessed by state-of-the-art unfitted methods \cite{hwx17,SFA18,BCDE21}, including arbitrary high-order accuracy with optimal convergence rates, uniform flux error bounds independent of discontinuous coefficients, and stiffness matrix condition numbers independent of the relative position of the interface to the meshes. We note that  the authors' earlier works \cite{WX10,hwx17}  focus solely on the unfitted IPFEM, unlike the present method which incorporates the DG method. Specifically, \cite{WX10} faces from unresolved numerical stability challenges: its stiffness matrix would become ill-conditioned in the presence of small cut elements, and the analysis requires flux jump penalization to retain optimal order $h$-convergence. Meanwhile, \cite{hwx17} adopts element merging to address the ill-conditioning issue but lacks a provably robust mesh-merging algorithm, leaving the practical implementation incomplete. 


The rest of our paper is organized as follows. In \cref{UIPDG-FEM}, we present the formulation of the UIPDG-FEM. Well-posedness and error estimates
for symmetric UIPDG-FEM are presented in \cref{SUIPDG-FEM}, including the error estimates of the solution in the energy and
$L^2$~norms and the error estimate of the flux in the $L^2$~norm. In \cref{condition-number}, we analyze the condition number of the discrete
system. In \cref{mergingSection}, we construct a  merging algorithm for the two dimensional case and prove its feasibility and reliability. In \cref{numerical-test}, we present some numerical experiments to confirm the theoretical findings.
		
Throughout this work, we let $C$ denote a generic positive constant that is independent of the mesh size $h$, the coefficient $\alpha$, the penalty parameters, and the position of the interface relative to the computational mesh. We employ the notation $A \lesssim B$ (or equivalently $B \gtrsim A$) to signify that $A \leq C B$, and $A \eqsim B$ to denote the two-sided inequality $A \lesssim B \lesssim A$.

\section{The unfitted interface penalty DG--FE Method}\label{UIPDG-FEM}
We now present a typical embedded interface configuration. To isolate the challenges posed by the interface, we assume that $\Omega$ can be conveniently discretized using Cartesian grids where the external boundary $\partial\Omega$ aligns with mesh element faces, while the interface $\Gamma$ remains immersed and may intersect elements arbitrarily. Let $\{\cTh\}$ be a family of conforming, quasi-uniform,  Cartesian partitions of $\Omega$. Let $h_K=\mathrm{diam}\, K$ for any $K\in\cTh$ and let $h=\max_{K\in\cTh}h_K$.  Define the interface element set ${\cal T}_\Gamma=\{K\in\cTh:\ K^\circ\cap\Gamma\neq\emptyset\}$, where $K^\circ$ denotes the element interior (all elements considered closed). Let $\cTm=\{K\in \cTh:\ K\cap \Omega_i\neq \emptyset\}$ for $i=1,2$, noting that $\cT_\Ga\subset\cTm$. Elements in $\cT_\Ga$ having insufficient overlap with $\Om_i$ may induce numerical instabilities in unfitted methods.

To address this, we merge ``small'' elements with adjacent neighbors to form geometrically regular macro-elements. To be precise, for a fixed threshold value $\de\in (0,1/2)$, we split $\cTm$ into two subsets:
\begin{align}\label{small}
  \cT_{h,i}^{small}&=\{K\in\cT_\Gamma:\;	
	  |K\cap\Om_i|<\de|K|\},\quad
	  \cT_{h,i}^{large}=\cTm \setminus\cT_{h,i}^{small},
\end{align}
for $i=1, 2$. Each small element $K\in\cT_{h,i}^{small}$ is merged with several neighboring elements  in $\cT_{h}$ including at least one from $\cT_{h,i}^{large}$ to construct a  macro-element $M_i(K)\supset K$. These macro-elements satisfy three critical conditions: the volume fraction $|M_i(K)\cap\Om_i|\gtrsim |M_i(K)|$, a size bound $\mathrm{diam}\, M_i(K)\lesssim h_K$, and pairwise disjointness such that for any distinct $K_1, K_2\in\cT_{h,i}^{small}$, either $M_i(K_1) = M_i(K_2)$ or $M_i(K_1)^{\circ}\cap M_i(K_2)^{\circ}=\emptyset$.
Under the interface-resolving assumptions (Assumptions \eqref{resolvingInterface}-\eqref{merging}), the merging algorithm in \cref{mergingAlgorithm} (for 2D cases) guarantees the success of the merging process. In particular, every small element $K\in\cTm^{small}$ is combined with at most three neighboring elements to form the merged macro-element $M_i(K)$ (illustrated in \cref{F1}). The merging algorithm for 3D cases will be developed in a future work.
For other elements in $\mathcal{T}_{h,i}^{large}$ that are not involved in the merging process, we define $M_i(K)$ as $K$ itself.  Then the induced mesh of $\cTm$ is defined as $\cMm~=\{M_i(K):K\in\cTm\},~i=1,2$, and $\cMm^{\Ga}$ is the set of all elements in $\cMm$ that intersect with $\Ga$.

		
	\begin{figure}[htbp]
			\centering
			\subfloat[$\cM_{h,1}$]{
				\includegraphics[scale=0.6]{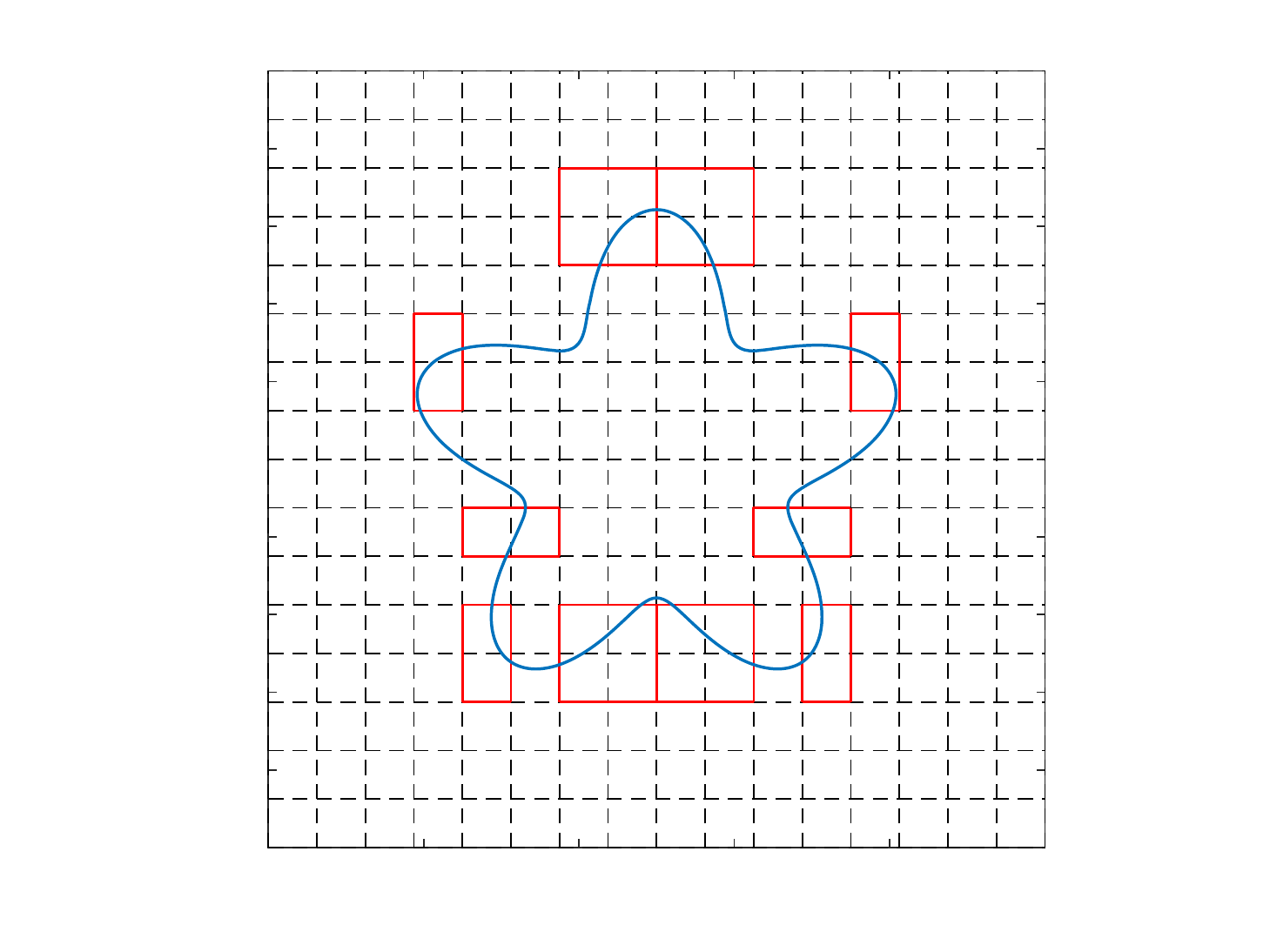}
			}
			\subfloat[$\cM_{h,2}$]{
				\includegraphics[scale=0.6]{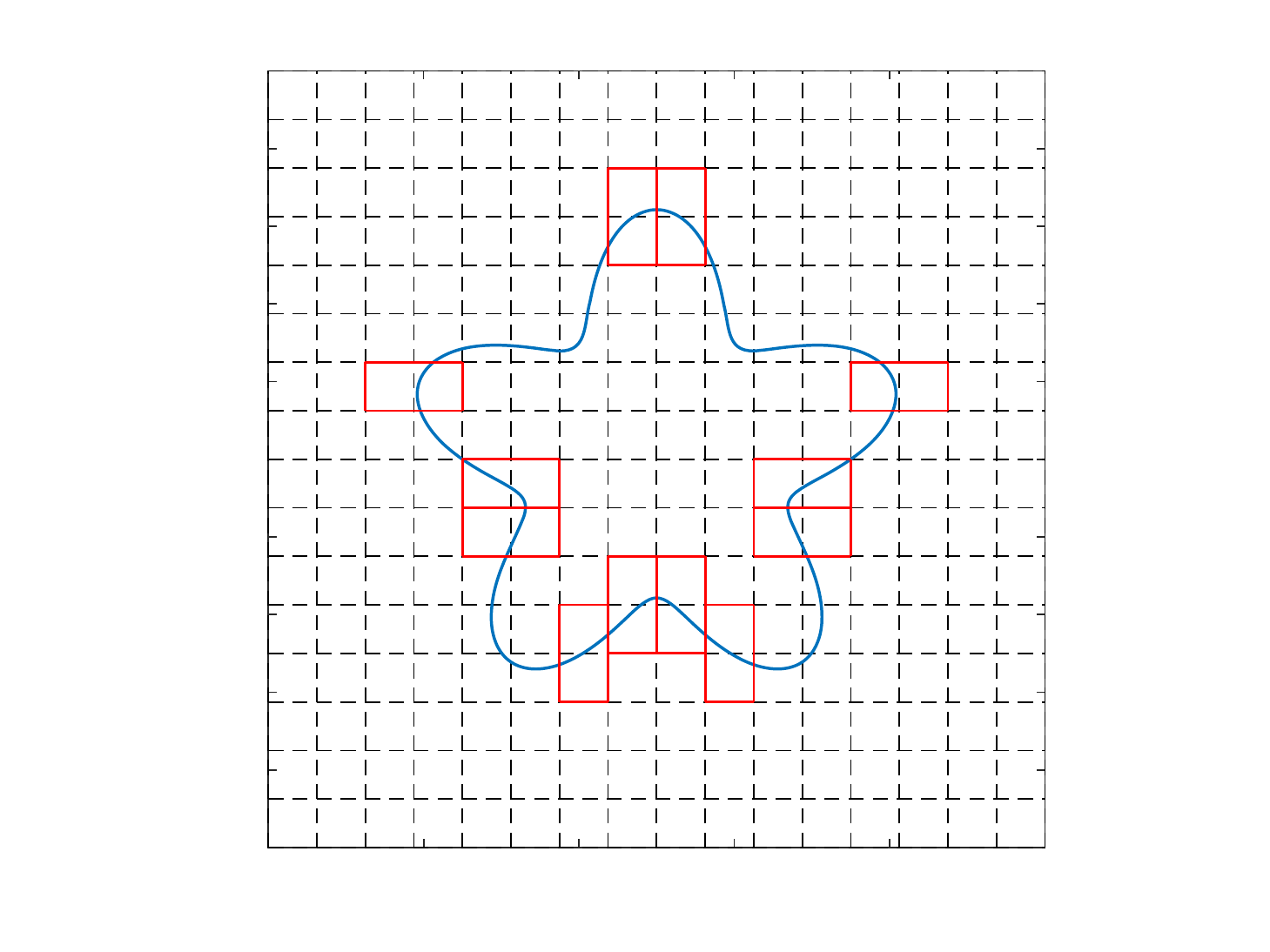}
			}
			\caption{Merged elements for $h=1/16$.}\label{F1}
		\end{figure}

 Let $\Om_{h,i}$ and $\widehat{\Om}_{h,i}$ be the union of elements in $\cMm$ and in $\cMm\cap\cTm$, respectively, we introduce the ``energy'' space,
		\begin{align*}
			V:= \{v:\; &v|_{K\cap\Om_i}\in H^2(K\cap\Om_i),\forall K\in\cMm,\; i=1,2;\;\\
			&v|_{\widehat{\Om}_{h,i}\cap\Om_i}\in H^1(\widehat{\Om}_{h,i}\cap\Om_i),\; i=1,2;\; v|_{\p\Om}=0 \}.
		\end{align*}

    For any $K\in\cTh$,
	let $\cQ_p(K)$ be the set of all polynomials whose degrees in each variable are less than or equal to $p$.
    The standard continuous finite element space on $\cTh$ is
	\[U_h:=\left\{v_h\in H_0^1(\Om):~v_h|_K\in\cQ_p(K),~\forall K\in\cTh\right\}.\]
	Let $U_{h,i}=U_h\cdot \chi_{\widehat{\Om}_{h,i}}$ be the finite element space that is continuous on non-merged elements for $i=1,2$.
    The space of discontinuous piecewise polynomials on merged elements is
		\[\setlength\abovedisplayskip{3pt}
		\setlength\belowdisplayskip{3pt}
		W_{h,i}:=\{v_h\in L^2\big(\cup_{K\in\cMm\setminus\cTm} K\big):~v_h|_{K}\in\cQ_p(K)\},\quad i=1,2.\]
		We define the composite spaces $V_{h,i}$ as the direct sum of $W_{h,i}$ and $U_{h,i}$ for $i=1,2$, and $V_h$ is
		\bn\label{Vh}
			V_h:=\left\{v_h=v_{h,1}\cdot\chi_{\Om_1}+v_{h,2}\cdot\chi_{\Om_2},v_{h,i}\in V_{h,i},~i=1,2\right\}.
		\en
		Clearly, $V_h\subset V$. Lastly, we introduce sets of faces
		\begin{align*}
        \mathcal{E}_h^i = \set{e:~e\subset\partial K, e\cap\Om_i\neq\emptyset , ~\forall ~K\in\cMm\setminus\cTm}, \;\mathring{\cE_h^i} = \set{e\cap\Om_i:~e\in\mathcal{E}_h^i},\; i=1,2.
		\end{align*}
        for boundaries of merged elements and their restrictions on sub-domains, respectively.

		In our numerical schemes, the quantities at the discontinuities are taken as the mean: Given any function $v \in V$, define
        \[\{ v\}_w:=w_1 v_1 + w_2 v_2,\quad\{ v\}^w:=w_2 v_1 + w_1 v_2,	\quad (v_i=v|_{\Om_i}, i=1, 2),\]
        where weights $w=(w_1,w_2)$ is a pair of nonnegative constants satisfying $w_1+w_2=1$.
		For internal discontinuities within subregions, we apply the arithmetic mean, i.e. $w_1=w_2=\frac12$. At the interface, to ensure robustness under coefficient contrast, we adopt the so-called ``harmonic weights'' from DG literature \cite{esz09,CYZ11}:
		\be\label{harmonic-weights}
		w_1=\frac{\al_2}{\al_1+\al_2},\; w_2=\frac{\al_1}{\al_1+\al_2}.
		\ee
		Clearly, $\{\al\}_w=\frac{2\al_1\al_2}{\al_1+\al_2}$ is the harmonic average of $\al_1$ and $\al_2$, and
		\begin{align}\label{emuw}
		\al_{min}\leq\{\al\}_w\leq\al_{max}, \; \{\al\}_w\leq2\al_{min},
		\end{align}
		where $\al_{min}=\min\{\al_1,\al_2\}, \al_{max}=\max\{\al_1,\al_2\}$.

		Testing the equation $\cref{elliptic-interface}_a$ by any $ v\in V$ and applying integration by parts with the identity $[ab]=\{a\}_w[b]+[a]\{b\}^w$, we derive
		\bn
			\int_{\Omega_1\cup\Omega_2}\al\nabla u\cdot\nabla v -\int_{\Gamma\cup\mathring{\mathcal{E}_h^1}\cup\mathring{\mathcal{E}_h^2}} \{\al\nabla u\cdot\bbn\}_w[ v]=\int_{\Omega_1\cup\Omega_2} f v +\int_\Gamma  g_N\{ v\}^w.
		\en
		Define the bilinear and linear forms:
		\begin{align}\label{ah}
			a_h(u, v)=& \int_{\oo}\al\nabla u\cdot\nabla v+\int_{\Ga\cup\mathring{\mathcal{E}_h^1}\cup\mathring{\mathcal{E}_h^2}} \frac{\ga\{\al\}_w}{h}[u][ v]\\
			&-\int_{\Ga\cup\mathring{\mathcal{E}_h^1}\cup\mathring{\mathcal{E}_h^2}} \big( \{\al\nabla u\cdot\bbn\}_w[ v] +\beta[u]\{\al\nabla v\cdot\bbn\}_w \big),\nn
		\end{align}
		\begin{align}\label{fh}
			f_h( v)=&\int_{\oo} f v +\int_\Ga  g_N\{ v\}^w -\beta\int_\Ga g_D\{\al\nabla v\cdot\bbn\}_w +\int_\Ga\frac{\ga\av{\al}_w}{h}g_D[ v],
		\end{align}
		where $\beta$ is a real number and $\gamma$ is a stabilization parameter to be specified later.
		
		The solution $u$ to the problem \cref{elliptic-interface} satisfies the following equation:
		\be\label{Variational-Pro}
			a_h( u, v)=f_h( v),\qquad  \forall v\in  V.
		\ee
		Our {\it unfitted interface penalty DG-FE Method} seeks $u_h\in  V_h$ such that
		\be\label{Discrete-Pro}
		\setlength\abovedisplayskip{2pt}
		\setlength\belowdisplayskip{2pt}
			a_h( u_h, v_h)=f_h( v_h),\qquad  \forall v_h\in  V_h.
		\ee
		From the equations \cref{Variational-Pro} and \cref{Discrete-Pro}, we have the Galerkin orthogonality:
		\be\label{Lem-Weak-Galerkin-orthogonality-0}
			a_h( u- u_h,v_h)=0,\qquad\forall v_h\in V_h.
		\ee

		\begin{remark} {\rm (i)}
			The discontinuity treatment follows established interior penalty methods from discontinuous/continuous Galerkin frameworks (see, e.g., \cite{BE07, FW09}).
The symmetric case $\beta = 1$ corresponds to a Symmetric unfitted interface penalty DG-FEM (SUIPDG-FEM), while $\beta \neq 1$ yields a nonsymmetric formulation.
We focus on $\beta = 1$ for simplicity of presentation.
For other values of $\beta$ in the bilinear form $a_h(\cdot,\cdot)$, the analysis differs only in the choice of the stability parameter $\gamma$; we omit the details for brevity.
		
			{\rm (ii)} In \cref{ah}, we specifically applied DG to the merged interface elements. However, our analysis remains valid if DG is utilized across all interface elements.
		\end{remark}

		\section{Error estimates of SUIPDG-FEM}\label{SUIPDG-FEM}

		To consider the boundedness and stability of the primal form $a_h(\cdot,\cdot)$, we define the broken $H^1$-norm on the space $V$:
		\setlength\abovedisplayskip{3pt}
		\setlength\belowdisplayskip{3pt}
		\begin{align}\label{norm}
			\norme{v}^2=\big\|\al^{\frac{1}{2}}\nabla v\big\|^2_{0,\oo}+\frac{\ga\{\al\}_w}{h}\norm{[v]}_{0,\Ga\cup\mathring{\mathcal{E}_h^1}\cup\mathring{\mathcal{E}_h^2}}^2
			+\frac{h}{\gamma \{\al\}_w }\norm{\{\al\nabla v\cdot\bbn\}_w}_{0,\Ga\cup\mathring{\mathcal{E}_h^1}\cup\mathring{\mathcal{E}_h^2}}^2.
		\end{align}
Throughout this paper, we use standard notation for Sobolev spaces. For instance, the norm on $H^s(\omega)$ is denoted by $\|\cdot\|_{s,\omega}$, and the corresponding seminorm by $|\cdot|_{s,\omega}$. The broken Sobolev space of the form $H^s(\Omega_1)\times H^s(\Omega_2)$ is denoted by $H^s(\oo)$ and is endowed with the norm $\norm{\cdot}_{s,\oo}=\big(\norm{\cdot}_{s,\Om_1}^2+\norm{\cdot}_{s,\Om_2}^2\big)^{\frac12}$, analogously for seminorms.

		Next, we recall key results and prove the approximation properties of the discrete spaces. 		
		The following trace and inverse trace inequalities on interface-cut elements play an important role in the analysis
		of our method; see \cite[Lemma 3.1]{WX10} for details of the proofs.
		\begin{lemma}\label{Lem-WX10}
			There exists a constant $h_0>0$, depending only on the interface $\Gamma$ and the shape
			regularity of the meshes, such that for all $h\in(0,h_0)$, and every interface-cut element $K\in\cMm$, $i=1, 2$, the following estimates hold:
			\eq{ \|v\|_{0,\p(K\cap\Om_i)}\lesssim h_K^{-\frac{1}{2}}\|v\|_{0,K\cap\Om_i} +\|v\|_{0,K\cap\Om_i}^{\frac12}\|\nabla v\|_{0,K\cap\Om_i}^{\frac12}, \quad\forall v\in H^1(K),\label{Lem-WX10-0}}
			\eq{ \|v_h\|_{0,\p(K\cap\Om_i)}\leq C_{\mathrm{tr}} h_K^{-\frac{1}{2}}\|v_h\|_{0,K\cap\Om_i}, \quad\forall v_h\in\cQ_p(K).\label{Lem-WX10-1}}
		\end{lemma}

We want to show that the DG-FE space has optimal approximation quality for piecewise smooth functions $v\in H^p(\Omega_1\cup \Omega_2)$.
For this purpose, we construct an interpolant of $v$ by the nodal interpolants of $H^s$-extensions of $v_1$ and $v_2$ as follows. Let $s\ge 2$ and choose extension operators $E_1: H^s(\Omega_1)\mapsto H^s(\Omega)\cap H_0^1(\Omega)$ and $E_2:
			H^s(\Omega_2)\cap H_0^1(\Omega)\mapsto H^s(\Omega)\cap H_0^1(\Omega)$ such that
			\eq{(E_iv)|_{\Omega_i}=v\quad {\rm and} \quad \|E_iv\|_{s,\Omega}\lesssim \|v\|_{s,\Omega_i},\quad  i=1, 2.\label{extension}}
where $v\in H^s(\Omega_1)$ for $i=1$ and $v\in H^s(\Omega_2),\, v|_{\partial\Omega}=0$ for $i=2$.

Let $I_h w$ be the standard nodal interpolation of $w\in H^2(\Omega)$. For any piecewise $H^2$ function $v\in H^2(\oo)$, denote $E_i v$ by $\tilde{v}_i$ and define interpolation of $v\in V$ onto $V_h$ by
		\begin{align}\label{interpolant}
        \tilde I_h v = \chi_{\Om_1} I_h \tilde{v}_1 + \chi_{\Om_2} I_h \tilde{v}_2.
		\end{align}
		We present an approximation error bound for the space $V_h$.
		\begin{lemma}\label{bestapproximation}
			Suppose $0<h\leq h_0$. Then for all $v\in H^{p+1}(\oo)$ with $v|_{\partial\Omega} = 0$, we have
			\begin{align*}
				\inf_{v_h\in V_h}\norme{v-v_h}\lesssim(\ga+\ga^{-1})^\frac12 h^p\big|\al^{\frac12}v\big|_{p+1,\oo}.				
			\end{align*}
		\end{lemma}
		\begin{proof}
			The proof follows \cite[Lemmas 3.5 and 3.6]{hwx17}, and we sketch it for readability. From standard finite element interpolation theory \cite{Ci78, EG04} and \cref{extension}, for $i = 1, 2$ and $j = 0, 1, \ldots, p+1$, we have
			\begin{align}\label{etai1}
			\bigg(\sum_{K\in\cM_{h,i}}\norm{\tilde{v}_i-I_h \tilde{v}_i}_{j,K}^2\bigg)^{\frac12}\lesssim h^{p+1-j}\abs{\tilde{v}_i}_{p+1,\Om_{h,i}}\lesssim h^{p+1-j}\norm{v}_{p+1,\Om_i}.
			\end{align}
			Let $\eta = v-\tilde I_h v$ and $\eta _i= \eta |_{\Om_i}= \tilde{v}_i-I_h \tilde{v}_i$ for $i=1,2$. We estimate each term in $\norme{\eta}$. First,
			\be\label{Thm-Err-estimates-H1-6}
				\big\|\al^{\frac{1}{2}}\nabla\eta \big\|_{0,\Omega_1\cup\Omega_2}^2=\sum_{i=1}^2\big\|\al_i^{\frac{1}{2}}\nabla\eta_i
				\big\|_{0,\Om_i}^2\lesssim h^{2p}\big\|\al^{\frac{1}{2}} v\big\|_{p+1,\oo}^2.
			\ee
			Using \cref{Lem-WX10-1}, \cref{etai1}, and the inequality $\{\alpha\}_w \leq 2\alpha_i$ for $i = 1, 2$, we bound the jump terms:
			\begin{align}\label{Thm-Err-estimates-H1-7}
				\norm{[\eta]}_{0,\Ga\cup\mathring{\mathcal{E}_h^1}\cup\mathring{\mathcal{E}_h^2}}^2
				\lesssim \sum_{i=1}^2\sum_{K\in\cMm^{\Ga}}\norm{\eta_i}_{0,\p(K\cap\Om_i)}^2
                \lesssim  \frac{h^{2p+1}}{\{\al\}_w} \big\|\al^{\frac{1}{2}} v\big\|^2_{p+1,\Omega_1\cup\Omega_2}.
			\end{align}
			Applying \cref{Lem-WX10-1} and \cref{etai1} again, the flux term satisfies
			\begin{align}\label{Thm-Err-estimates-H1-8}
				\norm{\{\al\nabla \eta\cdot\bbn\}_w}_{0,\Ga\cup\mathring{\mathcal{E}_h^1}\cup\mathring{\mathcal{E}_h^2}}^2&\lesssim
            \sum_{i=1}^2\sum_{K\in\cMm^{\Ga}}\norm{\al_iw_i\nabla \eta_i}^2_{0,\p(K\cap\Om_i)}\\
				&\lesssim  \{\al\}_wh^{2p-1} \big\|\al^{\frac{1}{2}} v\big\|^2_{p+1,\Omega_1\cup\Omega_2}.\nn
			\end{align}
			Combining \cref{Thm-Err-estimates-H1-6}--\cref{Thm-Err-estimates-H1-8} yields
			\begin{align*}
				\normee{ v-\tilde I_h v}\lesssim\big(\ga+\ga^{-1}\big)^\frac12 h^p\big\|\al^{\frac12}v\big\|_{p+1,\oo},\quad\forall v\in
				H^{p+1}(\oo).
			\end{align*}
			Let $q = (q_1, q_2)$ be a pair of $\mathcal{Q}_p$ polynomials in each $\Omega_i$. Noting that $\tilde{I}_h(v + q) - q \in V_h$, we obtain
			\begin{align*}
			\inf_{v_h\in V_h}\norme{v-v_h}&\le\normee{v+q-\tilde I_h(v+q)}\lesssim\big(\ga+\ga^{-1}\big)^\frac12 h^p\big\|\al^{\frac12}(v+p)\big\|_{p+1,\oo}.
			\end{align*}
			The result then follows from the Deny-Lions lemma \cite{EG04}.
			\begin{align*}
			\end{align*}
		\end{proof}

		\subsection{$H^1$ and $L^2$ error estimate}
		
		The following lemma establishes the continuity and coercivity of the bilinear form $a_h(\cdot,\cdot)$.
		\begin{lemma}\label{Thm-ah} We have
			\be\label{Thm-ah-continuity}
				|a_h( u, v)|\leq2\norme{u}\norme{ v},\qquad \forall  u, v\in  V.
			\ee
			Suppose $0<h\leq h_0$.
			Then there exists a constant $\ga_0>0$ independent of $h$, the coefficient $\al$, and the mesh $\cM_h$, such that for
			$\gamma\geq \ga_0$
			\be\label{Thm-Ah-coercivity}
			  a_h( v_h, v_h)\geq\frac12\norme{v_h}^2,\qquad\forall  v_h\in V_h,
			\ee
			where the constant $h_0$ is from \cref{Lem-WX10}.
		\end{lemma}
		\begin{proof} The continuity \cref{Thm-ah-continuity} follows directly from the Cauchy-Schwarz inequality. For coercivity, starting from \cref{ah} and \cref{norm},
			\begin{align*}
			  a_h( v_h, v_h)
			  =&\norme{v_h}^2-2\int_{\Ga\cup\mathring{\mathcal{E}_h^1}\cup\mathring{\mathcal{E}_h^2}} \{\al\nabla v_h\cdot\bbn\}_w[v_h]-\frac{h}{\gamma \{\al\}_w }\norm{\{\al\nabla
			  v_h\cdot\bbn\}_w}_{0,\Ga\cup\mathring{\mathcal{E}_h^1}\cup\mathring{\mathcal{E}_h^2}}^2\\
			  \ge& \norme{v_h}^2-\frac{\gamma \{\al\}_w }{2h}\norm{[v_h]}_{0,\Ga\cup\mathring{\mathcal{E}_h^1}\cup\mathring{\mathcal{E}_h^2}}^2-\frac{3h}{\gamma \{\al\}_w
			  }\norm{\{\al\nabla v_h\cdot\bbn\}_w}_{0,\Ga\cup\mathring{\mathcal{E}_h^1}\cup\mathring{\mathcal{E}_h^2}}^2.
			\end{align*}
			 Using the trace inequality of \cref{Lem-WX10-1}:
			\begin{equation}\label{norm3rdterm}
			  \frac{h}{\gamma \{\al\}_w }\norm{\{\al\nabla v_h\cdot\bbn\}_w}_{0,\Ga\cup\mathring{\mathcal{E}_h^1}\cup\mathring{\mathcal{E}_h^2}}^2\leq \frac{C_{\mathrm{tr}}^2}{\ga}\big\|\al^{\frac{1}{2}}\nabla v_h\big\|^2_{0,\Omega_1\cup\Omega_2}.
			\end{equation}
			Combining these estimates, we derive that
			\begin{align*}
			  a_h( v_h, v_h)\ge& \norme{v_h}^2-\max\Big\{\frac12,\frac{3C_{\mathrm{tr}}^2}{\gamma}\Big\}\norme{v_h}^2.
			\end{align*}
			By choosing $\gamma_0 = 6C_{\mathrm{tr}}^2$, the coercivity result in \cref{Thm-Ah-coercivity} follows.
		\end{proof}

		The following C\'{e}a lemma shows the discrete solution error is controlled by the best approximation error from $V_h$.
		\begin{lemma}\label{Cea}
			Suppose $0<h\leq h_0$ and $\ga\ge\ga_0$. Then
			\begin{align*}
				\norme{u-u_h}\le 5\inf_{v_h\in V_h}\norme{u-v_h}.
			\end{align*}
		\end{lemma}
		\begin{proof} For any $v_h\in V_h$, let $\eta_h=u_h-v_h$.
			By \cref{Lem-Weak-Galerkin-orthogonality-0} and \cref{Thm-ah}:
			\begin{align*}
				\norme{\eta_h}^2\le &2 a_h(\eta_h,\eta_h)=2 a_h(u-v_h,\eta_h)
				\le 4\norme{u-v_h}\norme{\eta_h}.
			\end{align*}
			Therefore, $\norme{u-u_h}\le\norme{u-v_h}+\norme{\eta_h}\le 5\norme{u-v_h}$.
		\end{proof}

		The following $H^1$-error estimate is a direct consequence of \cref{bestapproximation} and \ref{Cea}.
		\begin{theorem}\label{Thm-Err-estimates-H1}
			Suppose $0<h\leq h_0$, $\ga\ge\ga_0$, and $u \in H^{p+1}(\Omega)$ solves \cref{elliptic-interface}. Then there holds the following error estimate
			\be\label{Thm-Err-estimates-H1-0}
				\norme{u- u_h}\lesssim\ga^\frac12 h^p\big|\al^{\frac12}u\big|_{p+1,\oo}.
			\ee
		\end{theorem}


		For $L^2$-error estimates, we employ the Aubin-Nitsche trick:
		\begin{theorem}\label{Thm-Err-estimates-L2}
			Under the conditions of \cref{Thm-Err-estimates-H1}, there holds
			\bn
 				\| u- u_h\|_{0,\Omega}\lesssim \ga h^{p+1}\al_{\min}^{-\frac12}\big|\al^{\frac{1}{2}} u\big|_{p+1,\Omega_1\cup\Omega_2}.
			\en
		\end{theorem}
		\begin{proof} Let $\varphi$ solve the adjoint problem:
			\begin{align}\label{auxiliary-problem}
  				\left \{\ba {ll}-\nabla \cdot (\al\nabla \varphi ) =  u- u_h \quad & {\rm in} \quad \Omega_1\cup\Omega_2, \\
      			\jm{\varphi}= 0 , \quad\big[(\al\nabla \varphi)\cdot\bbn\big]= 0 \quad & {\rm on} \quad \Gamma, \\
       			\varphi = 0 \quad & {\rm on} \quad \p\Omega.\ea\right.
  			\end{align}
			From regularity estimate \cite{hz07}:
			\be\label{Thm-Err-estimates-L2-1}
 				\big|\al \varphi \big|_{j,\Omega_1\cup\Omega_2}\lesssim \big\| u- u_h\big\|_{0,\Omega}, \quad j=1, 2.
			\ee
			Therefore, by \cref{bestapproximation},
			\be\label{Thm-Err-estimates-L2-3}
				\inf_{\varphi_h\in V_h}\norme{\varphi-\varphi_h}\lesssim\ga^\frac12 h\big|\al^{\frac12}\varphi\big|_{2,\oo}\lesssim\ga^\frac12 \al_{\min}^{-\frac12}
				h\big\| u- u_h\big\|_{0,\Omega}.
			\ee
			Testing \cref{auxiliary-problem} with $\eta = u - u_h$ and using \cref{Lem-Weak-Galerkin-orthogonality-0},
			\cref{Thm-ah-continuity},  and \cref{Thm-Err-estimates-L2-3}, we get
			\begin{align*}
 				 \|\eta\|^2_{0,\Omega}&=a_h( u- u_h, \varphi )=\inf_{\varphi_h\in V_h}a_h( u- u_h, \varphi -\varphi_h)\\
  				 &\lesssim \ga^\frac12\al_{\min}^{-\frac12} h\norme{u- u_h}\big\| u- u_h\big\|_{0,\Omega}.
			\end{align*}
			The result follows using \cref{Thm-Err-estimates-H1}.
		\end{proof}


		\subsection{Flux error estimate}\label{FLUX}

	We consider the approximation error of the flux $\al\nabla u_h$. To handle the jump term on ``small'' faces, we establish the following trace transfer lemma. For ease of presentation, we consider a 2D configuration in \cref{K1K2}, where $K_1 \cap \Omega_i$ is ``small'' while $K_2$ and $K'$ are ``large'' elements. Merging $K_1$ with $K'$ ensures $|M_i(K_1)\cap\Om_i|\eqsim |M_i(K_1)| > |K_1|$.
	\begin{lemma}\label{traceTransform}
		For any $v_h\in V_h$, there holds
		\[\|[v_h]\|_{0,PP_1}^2\lesssim h\|\nabla v_h\|_{0,M_i(K_1)\cup K_2\cup K_3}^2+\|[v_h]\|_{0,PP_3}^2.\]
	\end{lemma}
	\begin{figure}[!htbp]
		\centering
			\includegraphics[scale=.5]{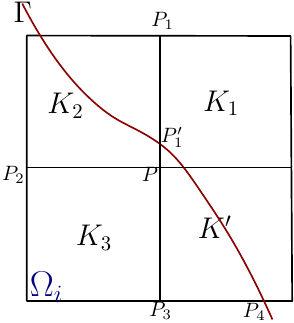}
		\caption{Transfer the trace from short edge to long edge.}
		\label{K1K2}
	\end{figure}
	\begin{proof}
		 Let $v_1 = v_h|_{M_i(K_1)}$, $v_2 = v_h|_{K_2}$, and $v_3 = v_h|_{K_3}$. Define the differences $w_i = v_i - v_1$ for $i=2,3$, where $v_1$ is naturally extended to adjacent elements. Considering the short edge $PP_1'=M_i(K_1) \cap M_i(K_2)$, we bound the jump term along this edge using terms that are confined to $\Omega_i$. In fact, we can derive that
		\begin{align*}
		\|[v_h]\|_{0,PP_1}^2&=\|w_2\|_{0,PP_1}^2\lesssim h\|\nabla w_2\|_{0,K_2}^2+\|w_2\|_{0,PP_2}^2\\
		&\lesssim h\|\nabla w_2\|_{0,K_2}^2+\|[v_h]\|_{0,PP_2}^2+\|w_3\|_{0,PP_2}^2\\
		&\lesssim h\|\nabla w_2\|_{0,K_2}^2+\|[v_h]\|_{0,PP_2}^2+h\|\nabla w_3\|_{0,K_3}^2+\|[v_h]\|_{0,PP_3}^2\\
		&\lesssim h(\|\nabla v_1\|_{0,M_i(K_1)}^2+\|\nabla v_2\|_{0,K_2}^2+\|\nabla v_3\|_{0,K_3}^2)+\|[v_h]\|_{0,PP_3}^2.
		\end{align*}
	\end{proof}

	We require a discrete
	extension operator $\mathfrak{E} _{h,i}$ mapping $V_{h,i}$
	onto $V_h$. Previous works \cite{bgss16} address the linear continuous case, and \cite{hwx17} extend these results to higher-order elements. For UIPDG-FEM, we construct the following discrete extension operator:
	\begin{lemma}\label{discrete-extension-operator}
		Suppose $0<h\leq h_0$. For $i=1,2$, there exists a discrete extension
		operator $\mathfrak{E}_{h,i}: V_{h,i}\mapsto V_h$ such that for any
		$v_{h,i}\in V_{h,i}$, $\mathfrak{E}_{h,i}v_{h,i}$ only differs from
		$v_{h,i}$ in $\Om_i$ by a constant $C$, $\mathfrak{E}_{h,i} v_{h,i}$ is continuous in $\Om_{3-i}$ and the following estimates hold:
		\begin{align*}
			\norm{[\mathfrak{E}_{h,i}v_{h,i}]}_{0,\Ga}\leq C_{\Ga}\norm{[v_{h,i}]}_{0,\mathcal{E}_h^i},
			\norm{\nabla\mathfrak{E}_{h,i} v_{h,i}}_{0,\Omega}^2\leq C_{\Om_i}^2(\norm{\nabla v_{h,i}}_{0,\Omega_i}^2+\frac{1}{h}\norm{[v_{h,i}]}_{\mathcal{E}_h^i}^2).
		\end{align*}
	\end{lemma}
	\begin{proof}
		Let $\{x_j\}_{j=1}^J$ be the nodal points of the space
		$U_h\cdot \chi_{\Om_{h,i}}$ and $\{\psi_j\}_{j=1}^J$ be the corresponding nodal basis functions. Firstly, we interpolate $v_{h,i}$ in the continuous
		space $U_h\cdot\chi_{\Om_{h,i}}$. Actually, we take the average value on $x_j$ when the nodal point is shared by several elements. This way,
		we construct an interpolant $I_{h,i}: V_{h,i}\mapsto U_h\cdot\chi_{\Om_{h,i}}$,
		\setlength\abovedisplayskip{3pt}
		\setlength\belowdisplayskip{3pt}
		\[(I_{h,i}v_{h,i})(x)=\sum_{j=1}^J\overline{v_{h,i}(x_j)}\psi_j(x),\qquad \forall v_{h,i}\in V_{h,i}.\]
		Note that $I_{h,i}v_{h,i}$ is continuous in $\Om_i$ and is only different from $v_{h,i}$ on $\mathcal{E}_h^i$. It is easy to check that
		\setlength\abovedisplayskip{3pt}
		\setlength\belowdisplayskip{3pt}
		\[\norm{\nabla I_{h,i}v_{h,i}}_{0,\Om_i}^2\leq\norm{\nabla v_{h,i}}_{0,\Om_i}^2+\frac1h\norm{[v_{h,i}]}_{0,\mathcal{E}_h^i}^2.\]
		Define $z_h = I_{h,i}v_{h,i} - C$ with $C = 0$ if $\partial\Omega_i \cap \partial\Omega \neq \emptyset$, else $C = \frac{1}{|\Omega_i|} \int_{\Omega_i} I_{h,i}v_{h,i}$. Using \cite[Lemma.3.3]{hwx17}, one can find a discrete extension operator
		$E_{h,i}:U_h\cdot\chi_{\Om_{h,i}}\mapsto U_h$ satisfying:
		\[E_{h,i}z_h=z_h~\text{in}~ \Om_{h,i},\quad \norm{E_{h,i}z_h}_{1,\Om}\leq C_{\Om_i} \norm{z_h}_{1,\Om_i}\leq C_{\Om_i}\norm{\nabla I_{h,i}v_{h,i}}_{0,\Om_i}.\]
		The required extension $\mathfrak{E}_{h,i}:V_{h,i}\mapsto V_h$ is then defined as
		\[\mathfrak{E}_{h,i}v_{h,i}=(v_{h,i}-C)\cdot\chi_{\Om_i}+E_{h,i}z_h\cdot\chi_{\Om_{3-i}}.\]
		Clearly, $\mathfrak{E}_{h,i}v_{h,i}$ is continuous in $\Om_{3-i}$, and the interface jump estimate follows from:
		\[\norm{[\mathfrak{E}_{h,i}v_{h,i}]}_{0,\Ga}^2\leq C_{\Ga} h^{d-1}\norm{v_{h,i}-I_{h,i}v_{h,i}}_{\infty,\mathcal{E}_h^i}^2\leq C_{\Ga} \norm{[v_{h,i}]}_{0,\mathcal{E}_h^i}^2.\]
	\end{proof}
	
	Combining this construction with the technique from \cite{bgss16}, we obtain our main flux error estimate:
	\begin{theorem}\label{Thm-Err-estimates-flux}
		Suppose $0<h\leq h_0$, $\ga\ge\ga_0$, and $ u\in
		H^{p+1}(\oo)$ with $u|_{\partial\Omega} = 0$ solves \cref{elliptic-interface}. Then there holds the following error estimate
		\be\label{Thm-Err-estimates-flux-0}
		  \norm{\al\nabla(u- u_h)}_{0,\oo}\leq C\ga^\frac12 h^p\big|\al u\big|_{p+1,\oo}.
		\ee
	\end{theorem}
	\begin{proof}
		Assume without loss of generality that  $\al_2\geq \al_1$. The flux error in $\Omega_1$ follows directly from \cref{Thm-Err-estimates-H1} since
		\begin{align}\label{estimate-in-Om1}
		  \big\|\al_1\nabla(u- u_h)\big\|_{0,\Omega_1}=\al_1^{\frac12}\big\|\al_1^{\frac12}\nabla(u- u_h)\big\|_{0,\Omega_1}
		  \lesssim\ga^\frac12 h^p\big|\al u\big|_{p+1,\Omega_1\cup\Omega_2}.
		\end{align}
		To bound $\big\|a_2\nabla(u- u_h)\big\|_{0,\Omega_2}$, we resort to the discrete extension operator
		$\mathfrak{E}_h:=\mathfrak{E}_{h,2}$ defined in \cref{discrete-extension-operator} and the interpolant $\tilde I_h$ in \cref{interpolant}. Substituting $v_h=\mathfrak{E}_h(u_h-\tilde I_h u)$
		into the weak orthogonality relation \cref{Lem-Weak-Galerkin-orthogonality-0} and multiplying by $\al_2$, we derive:
		\begin{align}
		  &\al_2^2\big\|\nabla(u_h-I_h u)\big\|_{0,\Omega_2}^2+\frac{\ga \al_2^2}{h}\norm{[u_h-I_h u]}_{0,\mathring{\cE_h^2}}^2\label{estimate-indentity}\\
		  &=\al_2\int_{\oo}\al\nabla (u-I_h u)\cdot\nabla v_h-\al_2\int_{\Omega_1}\al_1\nabla (u_h-I_h u)\cdot\nabla v_h\nn\\
		  &\quad -\al_2\int_{\Ga\cup\mathring{\cE_h^2}} (\{\al\nabla(u-u_h)\cdot\bbn\}_w[v_h]+[ u-u_h]\{\al\nabla v_h\cdot\bbn\}_w)\nn\\
		  &\quad -\al_1\al_2\int_{\mathring{\cE_h^1}}[u-u_h]\{\nabla v_h\cdot\bbn\}\nn\\
		  &\quad +\frac{\ga\al_2^2}{h}\int_{\mathring{\cE_h^2}}[u-I_h u][v_h]+\frac{\ga\{\al\}_w\al_2}{h}\int_{\Ga}[u-u_h][v_h].\nn
		\end{align}
		Here we have used the fact that $v_h$ is continuous in $\Om_1$, so the terms involving $[v_h]$ vanish within $\Om_1$.

		Using \cref{traceTransform}, \cref{discrete-extension-operator}, and the inequality \(ab \leq \varepsilon a^2 + \frac{b^2}{4\varepsilon}\), we bound each term on the right-hand side. For instance,
		\begin{align*}
		  &\al_2\int_{\oo}\al\nabla (u-I_h u)\cdot\nabla v_h\le\varepsilon \al_2^2\big\| \nabla v_h\big\|_{0,\Omega}^2+
			\frac1{4\varepsilon} \big\| \al\nabla (u-I_h u)\|_{0,\oo}^2\\
		  &\hskip 10pt\le\varepsilon C_{\Omega_2}^2\al_2^2 \Big(\big\|\nabla(u_h-I_h u)\big\|_{0,\Omega_2}^2+\frac1h\norm{[u_h-I_h u]}_{0,\mathring{\cE_h^2}}^2\Big)+
			  \frac1{4\varepsilon} \big\| \al\nabla (u-I_h u)\|_{0,\oo}^2,\\
		  &\al_2\int_\Ga [u-u_h]\{\al\nabla v_h\cdot\bbn\}_w\le \varepsilon h \al_2^2\norm{\{\nabla v_h\cdot\bbn\}}_{0,\Ga}^2+\frac{\{\al\}_w^2}{4\varepsilon h}
			  \norm{[ u-u_h]}_{0,\Ga}^2\\
		  &\hskip 10pt\le\varepsilon C_{\mathrm{tr}}^2C_{\Omega_2}^2\al_2^2 \Big(\big\|\nabla(u_h-I_h u)\big\|_{0,\Omega_2}^2+\frac1h\norm{[u_h-I_h u]}_{0,\mathring{\cE_h^2}}^2\Big)+
			  \frac{\{a\}_w}{4\varepsilon\gamma} \norme{u-u_h}^2.
		\end{align*}
		Similar bounds apply to interface jump terms using trace inequalities.

        Combining these estimates with \cref{estimate-indentity}, \cref{etai1}, \cref{Thm-Err-estimates-H1-0}, and \cref{emuw}, we obtain:
		\begin{align*}
			&\hskip 15pt\Big\{1-\varepsilon\Big(\frac{2C_{\Om_2}^2+C_{\mathrm{tr}}^2C_{\Om_2}^2}{\ga}+2C_{\Ga}^2+1\Big)-\frac{1}{2\varepsilon\ga}\Big\}\frac{\ga\al_2^2}{h}\norm{[u_h-I_h u]}_{0,\mathring{\cE_h^2}}^2\\
			&+\big(1-2\varepsilon(C_{\Om_2}^2+C_{\mathrm{tr}}^2+C_{\mathrm{tr}}^2C_{\Om_2}^2)\big)\al_2^2\norm{\nabla(u_h-I_h u)}_{0,\Om_2}^2  \lesssim\ga^{\frac12}h^p|\al_2 u|_{p+1,\oo}.
		\end{align*}
		Choosing $\varepsilon=\frac{1}{4(C_{\Om_2}^2+C_{\mathrm{tr}}^2+C_{\mathrm{tr}}^2C_{\Om_2}^2+C_\Ga^2+1)}$, and assuming sufficiently large $\ga$, we have
		\[\norm{\al_2\nabla(u_h-I_h u)}_{0,\Om_2}\lesssim\ga^{\frac12}h^p|\al_2 u|_{p+1,\oo}.\]
		The error estimate for $\big\|\al_2\nabla(u- u_h)\big\|_{0,\Omega_2}$ follows from \cref{etai1}
		and the triangle inequality.
		This completes the proof of \cref{Thm-Err-estimates-flux-0}.
	\end{proof}


	Furthermore, when $g_D\neq 0,~g_N\neq 0$, we have the following regularity estimates for the interface problem.

	\begin{theorem}\label{Thm:regularityFlux}
		Suppose that both $\p\Om$ and $\Ga$ are smooth and $f\in H^{s-2}(\oo)$, $g_D\in H^{s-\frac12}(\Ga)$, $g_N\in H^{s-\frac32}(\Ga)$ with $s\ge 2$, the interface problem \cref{elliptic-interface} satisfies the shift estimate,
		\be\label{regularityFlux}
			\abs{\al u }_{s,\oo}\le C\big(\norm{f}_{s-2,\oo}+\min\{\al_1,\al_2\}\norm{g_D}_{s-\frac12,\Ga}+\norm{g_N}_{s-\frac32,\Ga} \big).
		\ee
	\end{theorem}

	\begin{proof} To handle inhomogeneous interface conditions, we construct an auxiliary function \(\tilde{u}\) in two cases:  
		
		\textbf{Case I}($\alpha_1 \leq \alpha_2$): Solve the biharmonic problem in $\Om_1$,
		\begin{align*}
			\left\{\ba {ll} \Delta^2 \tilde{u}_1=0\quad &{\rm in} \quad\Om_1,\\
			\tilde{u}_1=g_D,\quad \al_1\nabla \tilde{u}_1\cdot\bbn= g_N \quad & {\rm on} \quad \Gamma,\ea\right.
		\end{align*}
		which admits a unique solution satisfying $\norm{\tilde{u}_1}_{s,\Om_1}\le C\Big(\norm{g_D}_{s-\frac12,\Ga}+\frac{1}{\al_1}\norm{g_N}_{s-\frac32,\Ga}\Big)$ \cite{GR86}. Extend \(\tilde{u}_1\) by zero to \(\Omega\).  
		
		\textbf{Case II}($\alpha_1 > \alpha_2$):  Solve the biharmonic problem in $\Om_2$,
		\begin{align*}
			\left\{\ba {ll} \Delta^2 \tilde{u}_2=0\quad &{\rm in} \quad\Om_2,\\
			\tilde{u}_2=-g_D,\quad \al_2\nabla \tilde{u}_2\cdot\bbn= -g_N \quad & {\rm on} \quad \Gamma,\\
			\tilde{u}_2=0,\quad \nabla \tilde{u}_2\cdot\bbn= 0 \quad & {\rm on} \quad \p\Om.\ea\right.
		\end{align*}
		yielding $\norm{\tilde{u}_2}_{s,\Om_2}\le C\Big(\norm{g_D}_{s-\frac12,\Ga}+\frac{1}{\al_2}\norm{g_N}_{s-\frac32,\Ga}\Big)$. Extend $\tilde{u}_2$ by zero to $\Omega$.
		
		Define $\hat{u}=u-\tilde{u}$ for both cases. This transforms the original problem into: 
		\begin{align*}
			\left\{\ba {ll} -\nabla \cdot (\al\nabla \hat{u} ) = f + \nabla\cdot(\al\nabla\tilde{u}) , \quad & {\rm in}\quad \Om,\\
			\jm{\hat{u}}=0,\quad \jm{\al\nabla \hat{u}\cdot\bbn}=0,\quad & {\rm on} \quad \Ga,\\
			\hat{u}=0,\quad & {\rm on}\quad \p\Om.\ea\right.
		\end{align*}
		Using regularity estimates for homogeneous jump conditions \cite{cgh10}, we obtain:  
\eqn{
\abs{\al u}_{s,\oo}\le C(\norm{f}_{s-2,\oo}+\norm{\al\tilde{u}}_{s,\oo}).}
Combining this with the bounds for $\tilde{u}_1$ and $\tilde{u}_2$ completes the proof of \eqref{regularityFlux}.
	\end{proof}

Using \cref{Thm-Err-estimates-flux}--\cref{Thm:regularityFlux}, we obtain the flux error bound independent of the coefficient jump:
		\begin{align}\label{rem-coefficient}
			\norm{\al\nabla(u-u_h)}_{0,\oo}\lesssim \ga^\frac12 h^p.
		\end{align}


	\section{Estimate of the condition number}\label{condition-number}

	In this section, we establish that the stiffness matrix condition number satisfies an upper bound consistent with standard finite element methods \cite{eg06}, independent of the position of the interface relative to the computational mesh.

	Let $N$ denote the dimension of the space $V_h$. Recall that the DG-FE approximation in \cref{Discrete-Pro} requires solving a linear system in $\mathbb{R}^N$:
        \[ A \bf u =  \bf F ,\]
    where ${\bf u} = (u_h(z_j))_{1\le j\le N}$, ${\bf F} = (f_h(\varphi_j))_{1\le j\le N}$, and $(\varphi_j)_{1\le j\le N}$ is the canonical basis of $V_h$. The stiffness matrix $A$ is symmetric and given by $A_{ij}=a_h(\varphi_i,\varphi_j)$.

    For any $v_h\in V_h$, let $\bv = (v_h(z_j))_{1\le j\le N}$ denote its nodal values. The mass matrix $M$ is given by $\bv^TM\bv=\|v_h\|_{0,\Om}^2$.  The following lemma establishes the relationship between the $L^2$-norm $\|\cdot\|_{0,\Om}$ and the energy norm $\norme{\cdot}$ on $V_h$.
		\begin{lemma}\label{Lem-CN-Ah}
			Assume $\ga_0\le \ga\lesssim 1$ and $0<h\leq h_0$. Then for all $v_h\in V_h$,
			\begin{align*}
				\al_{min}^{1/2}\|v_h\|_{0,\Om}\lesssim\norme{v_h}\lesssim \al_{max}^{1/2}h^{-1}\|v_h\|_{0,\Om}.
			\end{align*}
		\end{lemma}
		\begin{proof}
			Following \cite[Lemma 2.1]{arnold82}, a Poincar\'{e}-type inequality via duality yields
			\be\label{Lem-Poincare}
				\|v\|_{0,\Om}\lesssim \norm{\nabla v}_{0,\oo}+\norm{[v]}_{0,\Ga}+h^{-\frac12}\big(\sum_{i=1}^2\norm{[v]}_{0,\cE_h^i}^2\big)^{\frac12}.
			\ee
			We derive an inverse inequality for $v_h\in V_h$:
			\begin{align}\label{Poincare-Inverse-0}
				\norme{v_h}^2&\lesssim (1+\ga^{-1})\big\|\al^{\frac{1}{2}}\nabla v_h\big\|^2_{0,\oo} +\frac{\gamma \{\al\}_w }{h}\norm{[v_h]}_{0,\Ga\cup\cE_h^1\cup\cE_h^2}^2\\
				&\lesssim (1+\ga^{-1})h^{-2}\big\|\al^{\frac{1}{2}}v_h\big\|^2_{0,\oo}+\gamma\sum_{i=1}^2h^{-2}\al_i\big\|v_h\big\|^2_{0,\Om_i}\nn\\
				&\lesssim  \big(\gamma^{-1}+\gamma\big)h^{-2} \|\al^\frac12 v_h\|^2_{0,\Om}.\nn
			\end{align}
			From \cref{Thm-ah}, we have
			\begin{align*}
				\norme{v_h}^2\lesssim a_h(v_h,v_h)\lesssim \norme{v_h}^2,\quad\forall v_h\in V_h.
			\end{align*}
			Combining \cref{Lem-Poincare}, \cref{norm}, and \cref{emuw}, we obtain
			\begin{align*}
				\|v_h\|_{0,\Om}^2\lesssim \al_{min}^{-1}\big(1+\ga^{-1}+\ga^{-1}h\big) \norme{v_h}^2,\quad\forall v_h\in V_h.
			\end{align*}
        This completes the proof.
		\end{proof}

    On the other side, for any $v_h\in V_h$, it is clear that
			\begin{align*}
				\sum_{i=1}^2\norm{v_{h,i}}_{0,\Om_{h,i}}^2\lesssim \norm{v_h}_{0,\Om}^2\le \sum_{i=1}^2\norm{v_{h,i}}_{0,\Om_{h,i}}^2,
				\quad\forall v_h\in V_h.
			\end{align*}
    Due to norm equivalence in $V_h$, we further have $\sum_{i=1}^2\norm{v_{h,i}}_{0,\Om_{h,i}}^2\eqsim h^d\bv^T\bv$. Then by \cref{Lem-CN-Ah} and the identity
 		\begin{align*}
			\frac{\bv^TA\bv}{\bv^T\bv}=\frac{\norme{v_h}^2}{\|v_h\|_{0,\Om}^2}\cdot\frac{\bv^TM\bv}{\bv^T\bv},\quad \forall v_h\in
			V_h\setminus\set{0},
		\end{align*}
    we have an estimate of the spectral condition number of the stiffness matrix $A$:
		\begin{theorem}\label{Thm-condition-number} Suppose $\ga_0\le \ga\lesssim 1$ and $0<h\leq h_0$. Then
		\be\label{Thm-condition-number-0}
		{\rm cond}(A)\lesssim \frac{\al_{max}}{\al_{min}}\frac{1}{h^2}.\ee
		\end{theorem}

	\section{A merging algorithm}\label{mergingSection}
		In this section, we propose a merging algorithm for the two dimensional case, which ensures success of the merging process under suitable assumptions.
		
		
		For simplicity, we assume the initial mesh $\cT_h$ (prior to merging) is a uniform Cartesian grid with square elements of size $h$. We impose three key assumptions on $\cT_h$:
		\begin{enumerate}[(I)]
			\item\label{resolvingInterface} For any element $K\in\cT_h$,
            \begin{itemize}
            \item At most one edge intersects the interface $\Gamma$ at two points (including endpoints);
            \item The remaining three edges collectively intersect $\Gamma$ at most twice.
            \end{itemize}
			\item\label{connection} For any $P \in \Gamma$, the ball $B(P, 2\sqrt{2}h)$ remains connected within each subdomain $\Omega_i$ ($i=1,2$).
			\item\label{merging} Let $\kappa_m = \max_{P\in\Gamma} |\kappa(P)|$ denote the maximum interface curvature. With $t = \kappa_m h \leq 1$, we require:
			\eq{T(t):=\frac{30t(t+2)}{100t+63}\le 1-2\de.\label{Tt}}
		\end{enumerate}	
		\begin{remark} {\rm (i)} In much of the existing literature (e.g., \cite{LLW03,chen2022}), a standard assumption requires that the interface intersects the edges of an interface element at no more than two distinct points and these intersections occur on separate edges. However, this condition may not hold for evolving interfaces and the configurations shown in the second and third subfigures of \cref{type1} arise naturally during mesh refinement. To address these cases, Assumption \eqref{resolvingInterface} generalizes conventional constraints by permitting up to 3-4 intersections per element, specifically addressing the latter two scenarios in \cref{type1}. While the interface can intersect a single element edge multiple times without introducing analytical complications, our implementation is currently limited to handling the cases outlined in \cref{elementType}.

			{\rm (ii)} Assumption \eqref{connection} ensures non-overlapping merged elements in $\mathcal{M}_{h,i}$ (see Algorithm \ref{mergingAlgorithm} and Theorem \ref{thmmerging}). This is less restrictive than comparable conditions in \cite{chen2022}—Figure \ref{figA2} shows configurations permitted here but excluded in prior work.
			\begin{figure}[!htbp]
				\centering
				\begin{tikzpicture}[scale=0.3]
					\draw (0,0) grid (5,5);
					\draw [blue] (-1,1.05)--(6,1.05);
					\draw [blue] (-1,3.95)--(6,3.95);
					\node [blue,scale=0.7] at(6.5,1.05) {$\Ga$};
					\node [blue,scale=0.7] at(6.5,3.95) {$\Ga$};
				\end{tikzpicture}
				\caption{Merged elements do not overlap under Assumption \eqref{connection}.}\label{figA2}
			\end{figure}

			{\rm (iii)} According to Assumption \eqref{merging}, when $\de=0.25$, it is required that $\ka_m h\le 0.87$, and if $\de\le 0.22$, Assumption \eqref{merging} can be simplified to $\ka_m h\le 1$.
		\end{remark}

		Under Assumption~\eqref{resolvingInterface}, interface elements are classified into two primary types (Figure \ref{elementType}), with additional configurations generated through rotation and symmetry operations. For type-1 elements, the merging strategy is uniquely determined, as shown by the dotted square in \cref{type1}. In contrast, type-2 elements admit two possible merging strategies, illustrated by the dotted squares in \cref{type2}. The algorithm resolves type-2 ambiguities by comparing the lengths of two edge segments intersected by $\Omega_i$ (see Theorem \ref{feasibilityOfMerging}).
		\begin{figure}[!htbp]
			\centering
			\subfloat[type-1]{
				\label{type1}
				\begin{tikzpicture}[scale=0.7]
					\draw [black] (0,0) rectangle (1,1);
					\draw [blue] (-0.3,0.2)--(1.3,0.17);
					\draw [dashed] (0,-1) rectangle (1,0);
					\draw [black] (1.5,0) rectangle (2.5,1);
					\draw [blue] (2.8,-0.17) arc [start angle=60, end angle=120,radius=1.6];
					\draw [dashed] (1.5,-1) rectangle (2.5,0);
					\draw [black] (3,0) rectangle (4,1);
					\draw [blue] (2.75,0.17) arc [start angle=240, end angle=300,radius=1.5];
					\draw [dashed] (3,-1) rectangle (4,0);
					\node [black,scale=0.7] at (2.05,-0.7) {$\Om_i$};
				\end{tikzpicture}
			}
			\qquad
			\subfloat[type-2]{
				\label{type2}
				\begin{tikzpicture}[scale=0.7]
					\draw [black] (1,1) rectangle (2,2);
					\draw [dashed] (1,0) rectangle (2,1);
					\draw [dashed] (0,1) rectangle (1,2);
					\draw [blue,domain=24:66] plot ({4*cos(\x)+1.05-3.85/1.414},{4*sin(\x)+1.05-3.85/1.414});
					\node [black,scale=0.7] at (0.5,0.5) {$\Om_i$};

				\end{tikzpicture}
			}
			\caption{Two types of small elements.}
			\label{elementType}
		\end{figure}

		Before introducing the algorithm and establishing its feasibility, we first develop necessary theoretical foundations. The following lemma provides a quantitative relationship between a $C^2$-smooth curve and its tangent line, characterizing their local deviation.

		\begin{lemma}\label{Lem-C}
			Let $\mathcal{C}$ be a regular $C^2$ curve with maximum curvature $\ka_m$ and $P_0$ be a point on $\mathcal{C}$. For any point $P\in \mathcal{C}$, let $P_\perp$ be the foot of the perpendicular line from $P$ to the tangent at $P_0$. Assuming that $\ka_m|P_0P_\perp|\le 1$, then the distance $|PP_\perp|$ satisfies the following estimate:
			\eq{\label{dt}
			|PP_\perp|\le \ka_m^{-1}\big(1-\sqrt{1-(\ka_m|P_0P_\perp|)^2}\big).}
In other words, the part $\set{P\in \mathcal{C}:\; \ka_m|P_0P_\perp|\le 1}$ of the curve lies between two circles with radius $\ka_m^{-1}$ tangent at $P_0$ (see dashed lines in \cref{localCoordinate}).
		\end{lemma}
		\begin{figure}[!htbp]
			\centering
		\begin{tikzpicture}[scale=0.5]
			\draw [->] (-0.5,0)--(4,0);
			\draw [->] (0,-0.4)--(0,2);
			\draw[blue,domain=-0.1:3.8]plot({\x},{1/10*(\x)^2});
			\draw[blue,dashed,domain=-100:-40] plot({5*cos(\x)},{5*sin(\x)+5});
			\draw[blue,dashed,	domain=-100:-40] plot({5*cos(\x)},{-5*sin(\x)-5});
			\draw (3.3,1.089)--(3.3,0);
			\node[right,scale=0.5] at(3.3,1.089) {$P$};
			\node[below,scale=0.5] at(3.3,0) {$t$};
			\node[above,scale=0.5] at(3.1,-0.05) {$P_\perp$};
			\node[below,scale=0.5] at(-0.2,0.02){$P_0$};
			\node [right,scale=0.5] at(4,0) {$x$};
		\end{tikzpicture}
		\caption{Local deviation of $\mathcal{C}$ from its tangent line.}
		\label{localCoordinate}
		\end{figure}
		\begin{proof}
			Establish a local coordinate system (see \cref{localCoordinate})with $P_0$ as the origin and the tangent at $P_0$ aligned with the $x$-axis.Let $\mathcal{C}$ be represented by the function $y = f(x)$. Under this configuration, we obtain:
			\[1+\big(f'(t)\big)^2=\sec^2\int_0^t \frac{f''(x)}{1+\big(f'(x)\big)^2}\rd x=\sec^2\int_0^t \ka(x)\Big(1+\big(f'(x)\big)^2\Big)^{\frac12}\rd x.\]
			Let $N(t)=\max_{0\le x\le t} (1+(f'(x))^2)^{\frac12}$, it is easy to see
			\[N^2(t)\le\sec^2\int_0^t \ka_m N(x)\rd x,\quad N(t)\cos\big(\ka_m\int_0^t N(x)\rd x\big)\le 1.\]
			Next, let $M(t)=\int_0^t N(x)\rd x$, thus $M'(t)\cos(\ka_m M(t))\le 1$. Integrating with respect to $t$ from $0$ to $x$ yields that $\ka_m M(x)\le \arcsin(\ka_m x)$. Therefore we get
			\[N(t)\le \sec\big(\ka_m M(t)\big)\le \sec\big(\arcsin(\ka_m t)\big)=\big(1-(\ka_m t)^2\big)^{-\frac12}.\]
			From the integral form of the remainder in the Taylor expansion, we deduce that
			\begin{align*}
				|PP_\perp|=&|\int_0^t (t-x)f''(x)\rd x\le \int_0^t (t-x)\ka(x)\Big(1+\big(f'(x)\big)^2\Big)^\frac32 \rd x\\
				\le&\int_0^t (t-x)\ka_m\big(1-(\ka_m x)^2\big)^{-\frac32} \rd x=\ka_m^{-1}\big(1-\sqrt{1-(\ka_m t)^2}\big).
			\end{align*}
		\end{proof}
		The following lemma ensures that, under Assumptions \eqref{resolvingInterface}-\eqref{merging} on $\cT_h$, any ``small'' elements in $\cT_h$ has a proper ``large'' neighbouring element.

		\begin{lemma}\label{feasibilityOfMerging}
			Suppose the mesh $\cT_h$ satisfies  Assumptions~\eqref{resolvingInterface}--\eqref{merging}.
			\begin{enumerate}[(i)]
			\item For every $K \in \cTm^{small}$, there exists an element $K' \in \cTm^{large}$ such that $K$ and $K'$ can be merged, i.e., they share a common edge.
			\item For every $K'\in\cTm^{large}$, there are at most two distinct small elements in $\cTm^{small}$ that require merging with $K'$.
			\end{enumerate}
		\end{lemma}
		\begin{proof}
			We begin by considering $K$ as a type-1 small element and analyze its three representative cases illustrated in \cref{proofType1}. In the leftmost case of \cref{proofType1}, Assumption \eqref{connection} ensures that $K'$ lies entirely within $\Om_i$. For the middle case, let $A$ and $B$ denote the two intersection points between the bottom edge of $K$ and the interface $\Ga$. Define $P_0$ as the point on $\Ga$ between $A$ and $B$ that achieves the maximum distance $d$ to the secant line $\overline{AB}$. Under the condition $\ka_m h \leq 1$, it follows that $d \leq \left(1 - \frac{\sqrt{3}}{2}\right)h$, which implies $|K' \cap \Om_i| \geq \frac{\sqrt{3}}{2} |K'| > \delta |K'|$. Finally, in the rightmost case, the intersection satisfies $|K' \cap \Om_i| \geq \frac{\pi}{4} \ka_m^{-2} \geq \frac{\pi}{4} |K'| > \de |K'|$.
			\begin{figure}[h]
				\centering
				\subfloat[]{
					\label{proofType1}
					\begin{tikzpicture}[scale=0.75]
						\draw (0,0) grid (1,2);
						\draw [blue] (-0.2,1.1)--(1.2,1.2);
						\node [scale=0.6,right] at (0.3,0.5) {$K'$};
						\node [scale=0.6,right] at (0.3,1.5) {$K$};
						\draw (2,0) grid (3,2);
						\draw[blue,domain=1.8:3.2]plot({\x},{(\x-2.15)*(\x-2.85)/2+1});
						\node [scale=0.6,right] at (2.3,0.5) {$K'$};
						\node [scale=0.6,right] at (2.3,1.5) {$K$};
						\node [blue,above,scale=0.5] at(2.15,1) {$A$};
						\node [blue,scale=0.5,above] at (2.85,1) {$B$};
						\draw (4,0) grid (5,2);
						\draw[blue,domain=3.8:5.2]plot({\x},{-(\x-4.15)*(\x-4.85)/2+1});
						\node [scale=0.6,right] at (4.3,0.5) {$K'$};
						\node [scale=0.6,right] at (4.3,1.5) {$K$};
						\node [blue,above,scale=0.5] at(4.1,1.06125) {$A$};
						\node [blue,scale=0.5,above] at (4.9,1.06125) {$B$};
						\node [below,scale=0.6] at (2.5,-0.1) {$\Om_i$};
					\end{tikzpicture}
				}
				\qquad\qquad
				\subfloat[]{
					\label{proofType2}
					\begin{tikzpicture}[scale=0.75]
						\draw (-4,-1) grid (-2,1);
						\draw [blue,domain=20:70] plot ({2.2*cos(\x)-4.8},{2.2*sin(\x)-1});
						\node[blue,right,scale=0.5] at (-3.05,0.35) {$A$};
						\node[blue,right,scale=0.5] at (-2.9,0) {$B$};
						\node[left,scale=0.5] at(-2.95,0) {$C$};
						\node[left,scale=0.5] at (-4,-1) {$D$};
						\node[left,scale=0.5] at (-3.9,0) {$D_1$};
						\draw (-3.5,0.9763)--(-2.9,0.1226);
						\draw[red] (-3.8181,-0.3101) circle (1);
						\draw[red](-3.8181,-0.3101)--(-3,0.2649);
						\node[red,left,scale=0.5] at(-3.73,-0.3101) {$O$};
						\node[red,left,scale=0.4] at (-2.75,-0.1) {$B_1$};
						\node[red,left,scale=0.5] at (-3.9,0.6) {$D_2$};
						\node[right,scale=0.5] at(-3.2,0.57) {$l$};
						\draw (-4,-1)--(-3,0);
						\node[right,scale=0.6] at(-2.5,0.5) {$K$};
						\node[right,scale=0.6] at(-3.7,0.3) {$K'$};
						\node[right,scale=0.6] at(-4.6,-0.5) {$\Om_i$};
						\draw (-3.6291,0.673)--(-3.4,0.8340);
						\node[right,scale=0.5]at(-3.45,0.8340){$P_\perp$};
						\node[red,below,scale=0.5] at(-3.6291,0.7){$P'$};
						\node[blue,left,scale=0.5] at(-3.4,0.85) {$P$};
						\draw (-3.38,0.8056)--(-3.4209,0.7768)--(-3.4409,0.8052);
						\draw[->,red] (-3.7,-0.32)--(-3.5,-0.46);
						\node[blue,scale=0.6,left] at(-4,1){$\Ga$};
						\fill[red] (-3.5621,-0.5621) circle(0.5pt);
						\draw[red,dotted] (-3.5621,-0.5621)--(-3,0.2649);

						\draw (-1,-1) grid (1,1);
						\draw[red] (-0.5,-0.5) circle (1);
						\draw[dotted] (-1,-1)--(0,0);
						\draw  (0,-0.5)--(-0.5,-0.5)--(-0.5,0);
						\draw (0,-0.45)--(-0.05,-0.45)--(-0.05,-0.5);
						\draw (-0.55,0)--(-0.55,-0.05)--(-0.5,-0.05);
						\node[red,left,scale=0.5] at(-0.5,-0.5) {$O$};
						\fill[red] (-0.5,-0.5) circle(0.5pt);
						\node[left,scale=0.5] at(-0.48,-0.25){$s$};
						\node[below,scale=0.5]at(-0.25,-0.47){$s$};
						\node[left,scale=0.5] at(0.05,0) {$C$};
						\node[left,scale=0.5] at (-1,-1) {$D$};
						\node[left,scale=0.5] at (-0.9,0) {$D_1$};
						\node[red,left,scale=0.5] at (0.45,0) {$B_1$};
						\node[red,left,scale=0.5] at (-0.9,0.45) {$D_2$};
						\node[blue,right,scale=0.5] at (-0.1,0.4) {$A$};
						\draw[->] (1,0)--(1.2,0);
						\draw[->] (0,1)--(0,1.2);
						\node[below,scale=0.5] at (1.15,0){$x$};
						\node[right,scale=0.5] at(0,1.1) {$y$};
					\end{tikzpicture}
				}
				\caption{$K$ is merged with $K'$.}
			\end{figure}

			For the case in \cref{proofType2}, where $K$ is a type-2 small element, assume the edge segments $AC$ and $BC$ cut by $\Om_i$ satisfy $|AC| \geq |BC|$. Let $l$ denote the tangent to $\Ga$ at $A$, and let $\mathcal{O}$ be the closed disk of radius $\ka_m^{-1}$ tangent to $l$ at $A$, with its center $O$ lying within $\Om_i$. By \cref{Lem-C}, the inequalities $|PP_\perp| \leq |P'P_\perp|$ and $|B_1C| \leq |BC|$ imply $|K' \cap \Om_i| \geq |K' \cap \mathcal{O}|$.
Introduce the dimensionless parameters $a$, $b$, and $y$ such that $|AC| = a h$, $|B_1C| = b h$ (with $b \in [0, a]$), and $|D_1D_2| = y h$. For fixed $\ka_m$ and $a$, observe that $y$ decreases as $b$ increases. The intersection area is
			\[\mathbb{A}_1=|K'\cap\mathcal{O}|=\frac{(ah+yh)h}{2}+\ka_m^{-2}\theta(y)-\frac12\ka_m^{-2}\sin\big(2\theta(y)\big),\]
where $\theta(y) = \arcsin\left(\frac{\sqrt{1 + (y - a)^2}}{2 (\kappa_mh)^{-1}}\right)$. A direct calculation shows $\frac{\partial \mathbb{A}_1}{\partial y} \geq 0$, indicating $\mathbb{A}_1$ is minimized when $b=a$. Furthermore, as illustrated in the rightmost subfigure of \cref{proofType2}, $\mathbb{A}_1$ attains its global minimum when $a = b = 0$. Thus, by Assumption~\eqref{merging}, we conclude
			\begin{align*}
				|K'\cap\mathcal{O}|&\ge\frac{h^2}{t^2}\int_0^t\Big( \big(\frac12-\sqrt{2}(x-t)-(x-t)^2\big)^\frac12-\frac{\sqrt{2}}{2}\Big) \rd x\\
				&=\frac{1}{4t^2}+\frac{\pi}{8t^2}-(\frac{\sqrt{2}}{4t}-\frac12)\sqrt{\frac{1}{2t^2}+\frac{\sqrt{2}}{t}-1}-\frac{1}{2t^2}\arcsin(\frac{\sqrt{2}}{2}-t)-\frac{\sqrt{2}}{2t}\\
				&\ge \frac{1-T(t)}{2}|K'|\ge\de|K'|,
			\end{align*}
			the proof of the second inequality is tedious but trivial; we therefore omit it here for brevity.
			
 			\begin{figure}[htbp]
				\centering
				\subfloat[]{
					\label{case1}
					\begin{tikzpicture}[scale=0.5]
						\draw [black] (1,0) grid (4,2);
						\draw [blue,domain=5:175] plot ({1.4*cos(\x)+2.5},{1.4*sin(\x)-0.35});
						\draw [red,domain=0:180] plot ({1.2*cos(\x)+2.5},{1.2*sin(\x)-0.15});
						\node [blue,right,scale=0.5] at (1.05,0.5) {$\Ga$};
						\node [black,right,scale=0.5] at (2.3,1.6) {$K_1$};	
						\node [black,right,scale=0.5] at (1.5,0.3) {$K_2$};
						\node [black,right,scale=0.5] at (3.3,0.5) {$K_3$};
						\node[above,scale=0.5] at(2.5,1){$C$};
						\node [black,right,scale=0.5] at (2.3,0.5) {$K'$};
						\node [black,right,scale=0.5] at (2.3,-0.2) {$\Om_i$};
					\end{tikzpicture}
				}
				\qquad\qquad
				\subfloat[]{
					\label{case2}
					\begin{tikzpicture}[scale=0.5]
						\draw [black] (-7,0) grid (-5,2);
						\draw [blue] (-7.3,1.5)--(-5.5,-0.3);
						\node [black,right,scale=0.5] at (-6.73,0.5) {$K'$};
						\node [black,right,scale=0.5] at (-5.7,0.5) {$K_1$};
						\node [black,right,scale=0.5] at (-6.7,1.5) {$K_2$};
						\node [black,right,scale=0.5] at (-7.3,-0.2) {$\Om_i$};

						\draw (-3.5,0.5) rectangle (-0.5,1.5);	
						\draw (-2.5,0.5)--(-2.5,1.5);
						\draw (-1.5,0.5)--(-1.5,1.5);
						\draw [blue,domain=-3.1:-1] plot ({\x},{1.3-(\x+2.4)*(\x+1.7)/1.1});
						\node [black,right,scale=0.5] at (-3.4,1.1) {$K_1$};
						\node [black,right,scale=0.5] at (-2.2,0.9) {$K'$};
						\node [black,right,scale=0.5] at (-1.2,0.9) {$K_2$};
						\node [black,right,scale=0.5] at (-2.2,0.3) {$\Om_i$};	
					\end{tikzpicture}
				}				
				\caption{No more than two small elements in $\cTm^{small}$ require merging with $K'\in\cTm^{large}$.}	
			\end{figure}
				
			For the second part of the lemma, consider the scenario (illustrated in \cref{case2}) where multiple small elements merge with a single element. Suppose three small elements $K_1,K_2,K_3$ are merged with $K'$, as shown in \cref{case1}. Let $C$ denote the point on the interface $\Ga$ where the tangent line is parallel to the bottom edge of $K'$. Without loss of generality, assume $C$ lies closer to the left edge of $K'$. In the same way as in the proof of (i), thanks to Assumption~\eqref{merging}, we derive:
			\[|K_2\cap\Om_i|\ge \left(\frac12 a(t)b(t)+t^{-2}\theta_1(t)-\frac{t^{-2}}{2}\sin\big(2\theta_1(t)\big)\right)|K_2|:=T_1(t)|K_2|,\]
where
			\[a(t)=\sqrt{2t^{-1}-1}-\frac12,b(t)=\sqrt{t^{-2}-\frac14}-t^{-1}+1,\theta_1(t)=\arcsin\Big(t\frac{\sqrt{a(t)^2+b(t)^2}}{2}\Big).\]
			A direct calculation confirms $T_1(t)>T(t)\ge\de$, which is a contradiction.
		\end{proof}


		We now introduce a merging algorithm, which starts from the pairwise merging according to \cref{feasibilityOfMerging}, and ends at resolving cases where large elements require repeated merging.  This procedure inherently defines the macro-elements $M_i(K)$.
		\begin{algorithm}
			\caption{Merging algorithm}\label{mergingAlgorithm}
			\begin{algorithmic}[1]
				\REQUIRE Mesh satisfying Assumptions \eqref{resolvingInterface}--\eqref{merging}.
				\FOR{$i=1,2$}
				\STATE Let $N_i(K')=\emptyset,~\forall K'\in\cTm^{large}.$
				\FOR{$K\in\cTm^{small}$}
				\STATE Find a neighbor $K'\in\cTm^{large}$ for $K$ according to Lemma \ref{feasibilityOfMerging}, let $N_i(K')=N_i(K')\cup\{K\}$.
				\ENDFOR
				\FOR{$K'\in\cTm^{large},~\mathrm{card}(N_i(K'))>0$}
				\IF{$\mathrm{card}(N_i(K'))=1$}
				\STATE Define $M_i(N_i(K'))=M_i(K')=K'\cup N_i(K').$
				\ELSE
				\STATE Let $\widetilde{K}$ be the minimum rectangle containing all elements in $N_i(K')$ and $K'$,  define $M_i(K)=\widetilde{K}$, for any $K\in\cTm$ contained in $\widetilde{K}$.
				\ENDIF
				\ENDFOR
				\ENDFOR
			\end{algorithmic}
		\end{algorithm}

		The following theorem says that the mesh obtained by the merging algorithm is well-defined.
		\begin{theorem} \label{thmmerging}
		Suppose that the mesh $\cT_h$ satisfies  Assumptions~\eqref{resolvingInterface}--\eqref{merging}. Let $\cMh$ be the mesh obtained by the merging \cref{mergingAlgorithm}.  We have for $i=1, 2$, {\rm (i)} $\mathrm{diam}\, K\lesssim h$, for any $K\in\cMm$, {\rm (ii)}   $K_1^{\circ}\cap K_2^{\circ}=\emptyset$, for  any $K_1,K_2\in\cMm$.
		\end{theorem}
		\begin{proof}
			\begin{figure}[htbp]
				\centering
				\subfloat[]{
					\label{overlap1}
					\begin{tikzpicture}[scale=0.5]
						\draw (0,0) rectangle (2,2);
						\draw (1,1) rectangle (3,3);
						\draw [dashed] (0,1)--(2,1);
						\draw [dashed] (1,0)--(1,2);
						\draw [dashed] (1,2)--(3,2);
						\draw [dashed] (2,1)--(2,3);
						\draw [dashed] (1,0)--(3,2);
						\draw [blue] (-0.3,1.5)--(1.5,-0.3);
						\draw [blue] (1.5,3.3)--(3.3,1.5);
						\node [below,scale=0.5] at(1,0) {$C$};
						\node [right,scale=0.5] at(3,2) {$D$};
						\node [right,scale=0.6] at(0.3,1){$K_1$};
						\node [right,scale=0.6] at (2.2,2) {$K_2$};
						\node [right,scale=0.6] at (1.3,1.5) {$K'$};
						\node [right,scale=0.7] at (0,0.2) {$\Om_i$};
					\end{tikzpicture}
				}
				~
				\subfloat[]{
					\label{overlap2}
					\begin{tikzpicture}[scale=0.5]
						\draw (0,0) rectangle (2,2);
						\draw (1,0) rectangle (3,2);
						\draw [dashed] (0,1)--(3,1);
						\draw [blue] (-0.3,1.5)--(1.5,-0.3);
						\draw [blue] (1.5,2.3)--(3.3,0.5);
						\draw [dashed] (1,0)--(3,2);
						\node [below,scale=0.5] at(1,0) {$C$};
						\node [right,scale=0.5] at(3,2) {$D$};
						\node [right,scale=0.6] at(0.3,1){$K_1$};
						\node [right,scale=0.6] at (2.2,1) {$K_2$};
						\node [right,scale=0.6] at (1.3,0.5) {$K'$};
						\node [right,scale=0.7] at (0,0.2) {$\Om_i$};
					\end{tikzpicture}
				}
				~
				\subfloat[]{
					\label{overlap3}
					\begin{tikzpicture}[scale=0.5]
						\draw (-3.5,0.5) rectangle (-0.5,1.5);	
						\draw [dashed] (-2.5,0.5)--(-2.5,1.5);
						\draw (-1.5,0.5)--(-1.5,1.5);
						\draw [blue,domain=-3.2:-0.9] plot ({\x},{1-(\x+2.4)*(\x+1.7)/2});
						\draw (-1.5,0.5) rectangle (0.5,2.5);
						\draw [dashed] (-0.5,1.5)--(0.5,1.5);
						\draw [dashed] (-0.5,1.5)--(-0.5,2.5);
						\draw [blue] (-0.9,2.7)--(0.7,1.1);
						\draw [dashed] (-1.5,0.5)--(0.5,2.5);
						\node [black,right,scale=0.6] at (-3.1,1.1) {$K_1$};
						\node [black,right,scale=0.6] at (-0.5,1.5) {$K_2$};
						\node [blue,right,scale=0.6] at (-1.2,0.9) {$K'$};
						\node [black,right,scale=0.7] at (-2.7,0.3) {$\Om_i$};
						\node [below,scale=0.5] at(-1.5,0.5) {$C$};
						\node [right,scale=0.5] at(0.5,2.5) {$D$};
					\end{tikzpicture}
				}
				~
				\subfloat[]{
					\label{overlap4}
					\begin{tikzpicture}[scale=0.5]
						\draw (-2.5,0.5) rectangle (-0.5,1.5);	
						\draw [blue] (-1.6,1.7)--(-0.8,0.4);
						\draw (-1.5,0.5) rectangle (0.5,2.5);
						\draw [dashed] (-0.5,1.5)--(0.5,1.5);
						\draw [dashed] (-0.5,1.5)--(-0.5,2.5);
						\draw [blue] (-0.9,2.7)--(0.7,1.1);
						\draw [dashed] (-1.5,0.5)--(0.5,2.5);
						\node [black,right,scale=0.6] at (-2.3,1.1) {$K_1$};
						\node [black,right,scale=0.6] at (-0.5,1.5) {$K_2$};
						\node [blue,right,scale=0.6] at (-1.2,0.9) {$K'$};
						\node [black,right,scale=0.7] at (-3.1,1.1) {$\Om_i$};
						\node [left,scale=0.5] at(-1.5,0.5) {$C$};
						\node [right,scale=0.5] at(0.5,2.5) {$D$};
					\end{tikzpicture}
				}
				~
				\subfloat[]{
					\label{overlap5}
					\begin{tikzpicture}[scale=0.5]
						\draw (-2.5,0.5) rectangle (-0.5,1.5);	
						\draw [blue,domain=-55:55] plot ({0.7*cos(\x)-2.1},{0.7*sin(\x)+1});
						\draw (-1.5,0.5) rectangle (0.5,2.5);
						\draw [dashed] (-0.5,1.5)--(0.5,1.5);
						\draw [dashed] (-0.5,1.5)--(-0.5,2.5);
						\draw [blue] (-0.9,2.7)--(0.7,1.1);
						\draw [dashed] (-1.5,1)--(0.5,2.5);
						\node [black,right,scale=0.6] at (-2.3,1.1) {$K_1$};
						\node [black,right,scale=0.6] at (-0.5,1.5) {$K_2$};
						\node [blue,right,scale=0.6] at (-1.2,0.9) {$K'$};
						\node [black,right,scale=0.7] at (-3.1,1.1) {$\Om_i$};
						\node [left,scale=0.5] at(-1.5,1) {$C$};
						\node [right,scale=0.5] at(0.5,2.5) {$D$};
					\end{tikzpicture}
				}
				\caption{The presence of overlapping macro-elements contradicts Assumption \eqref{connection}.}	\label{overlap}
				\vspace{-10pt}
			\end{figure}
			
			By \cref{feasibilityOfMerging}(ii), we have $\mathrm{card}(N_i(K'))\le 2$, $\forall K'\in\cTm^{large}$, which implies that (i) holds. To verify (ii), suppose there exist two macro-elements $K_1, K_2 \in \cMm$ with overlapping interiors ($K_1^\circ \cap K_2^\circ \neq \emptyset$). Consider the first overlap scenario of \cref{case2}, illustrated in \cref{overlap}. In this case, we can always identify points $C$ and $D$ with $|CD|\le 2\sqrt{2}h$, such that there exists a point on $\Ga$ violating Assumption \eqref{connection}, leading to a contradiction. Other overlap configurations are similarly incompatible with Assumptions \eqref{resolvingInterface}-\eqref{merging} and can therefore be ruled out. Thus, condition (ii) also holds. 	
		\end{proof}

		\begin{remark} {\rm (i)}		In comparison to the algorithm presented in \cite{chen2022}, our algorithm exhibits several notable differences. Firstly, we allow for more kinds of interface element configurations (see e.g.  the 2nd, 3rd subfigures in \cref{type1} ). While our analysis remains valid under the connectivity condition in \cite[Definition 3.1]{chen2022}, we adopt the weaker Assumption \eqref{connection} to broaden applicability. A key distinction lies in the geometric criteria: our method enforces merging feasibility through an area-based threshold \cref{small} and curvature constraints (Assumption \eqref{merging}), whereas \cite{chen2022} relies on side-length criteria and chains of interface elements. Furthermore, in this work, the merging is performed separately for two subregions, which to some extent reduces the complexity of the algorithm.

			{\rm (ii)} Practically, merging is operationally realized through congruence transformations of the stiffness matrix. Let $\tilde{\mathcal{N}}_{h,1} = \{\tilde{z}_1^1, \ldots, \tilde{z}_1^{\tilde{N}_1}\}$ and $\tilde{\mathcal{N}}_{h,2} \setminus \partial\Omega = \{\tilde{z}_2^1, \ldots, \tilde{z}_2^{\tilde{N}_2}\}$ denote the nodal sets of the pre-merged DG-FE spaces, which inherit the discontinuities of $V_{h,1}$ and $V_{h,2}$. For $v_h \in V_h$, define the extended nodal vector $\tilde{\mathbf{v}} = \big(v_h(\tilde{z}_1^1), \ldots, v_h(\tilde{z}_1^{\tilde{N}_1}), v_h(\tilde{z}_2^1), \ldots, v_h(\tilde{z}_2^{\tilde{N}_2})\big)$. Merging corresponds to interpolating $\tilde{\mathbf{v}}$ onto the merged nodal vector $\mathbf{v}$ (defined in Section \ref{condition-number}) via a bounded linear operator represented by a matrix $B$ with $O(1)$ entries, such that $\tilde{\mathbf{v}} = B\mathbf{v}$. The merged stiffness matrix $A$ and load vector $\mathbf{F}$ are then derived from their unmerged counterparts $\widetilde{A}$ and $\widetilde{\mathbf{F}}$ as $A = B^T \widetilde{A} B$ and $\mathbf{F} = B^T \widetilde{\mathbf{F}}$.
			\end{remark}

			\section{Numerical tests}\label{numerical-test}

			In this section, we illustrate the convergence rates of the proposed methods and demonstrate the effectiveness of the merging algorithm by some numerical tests.  The algorithms are implemented in MATLAB.
			
			As in \cite{hwx17}, we transform the integral domain to a proper reference domain and then resort to the one dimensional
			Gauss quadrature rules. The influence of quadrature formulas as well as the approximation of the
			interface on the error estimates will be considered in a separate work. We refer to \cite{bhl17,L16} for analyses of
			geometry error for some unfitted finite element methods and refer to \cite{L16,S15,fo06} for works on
			computing integrals on the domain with curved boundary.

			\begin{example}\label{ex1}
				The domain $\Om$ is the unit square $(0,1)\times(0,1)$ and the interface $\Ga$ is a 5-petal flower
				centered at $(0.5,0.5)$. The interface is defined in polar coordinates $(R, \theta)$ relative to the pole at $(\frac{1}{2}, \frac{1}{2})$ by the equation: $R=\frac{1}{2}+\frac{1}{7} \sin 5\theta$. The domain $\Om$ is partitioned using uniform Cartesian grids with mesh size $h$, as illustrated in Fig.~\ref{F1}. The source term $f$, Dirichlet data $g_D$, and Neumann data $g_N$ are prescribed such that the exact solution to the interface problem \cref{elliptic-interface} is the following piecewise smooth function:
				\begin{align}
				u(x,y)=\begin{cases}
				\frac{1}{\al_1}\exp(xy),&(x,y)\in\Om_1,\\
				\frac{1}{\al_2}\sin(\pi x)\sin(\pi y),~&(x,y)\in\Om_2,
				\end{cases}
				\end{align}
				where $\al_1=1000, \al_2=1$ or $\al_1=1,\al_2=1000.$
			\end{example}
			\begin{figure}[htbp]
				\centering
				\subfloat[]{
					\includegraphics[scale=0.26]{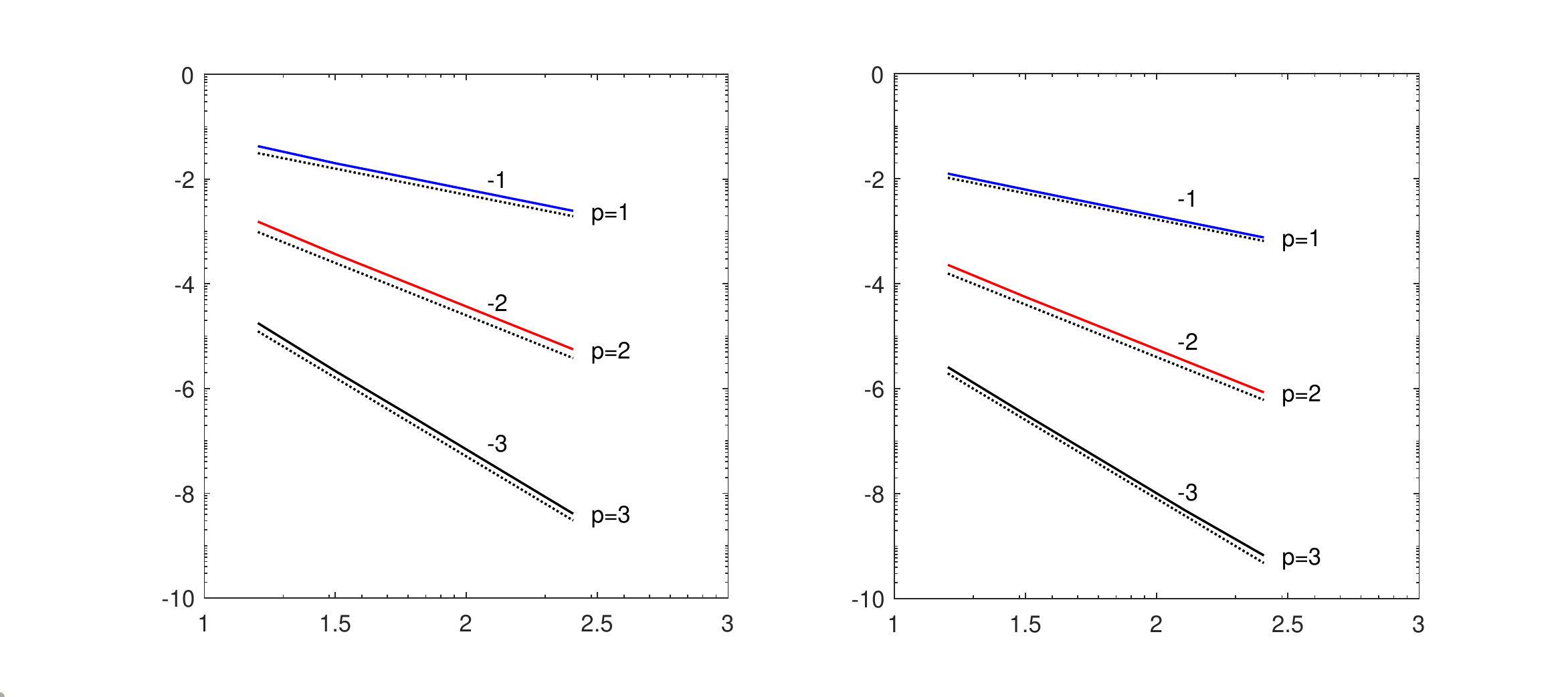}\label{F6}
				}
				\subfloat[]{
					\includegraphics[scale=0.26]{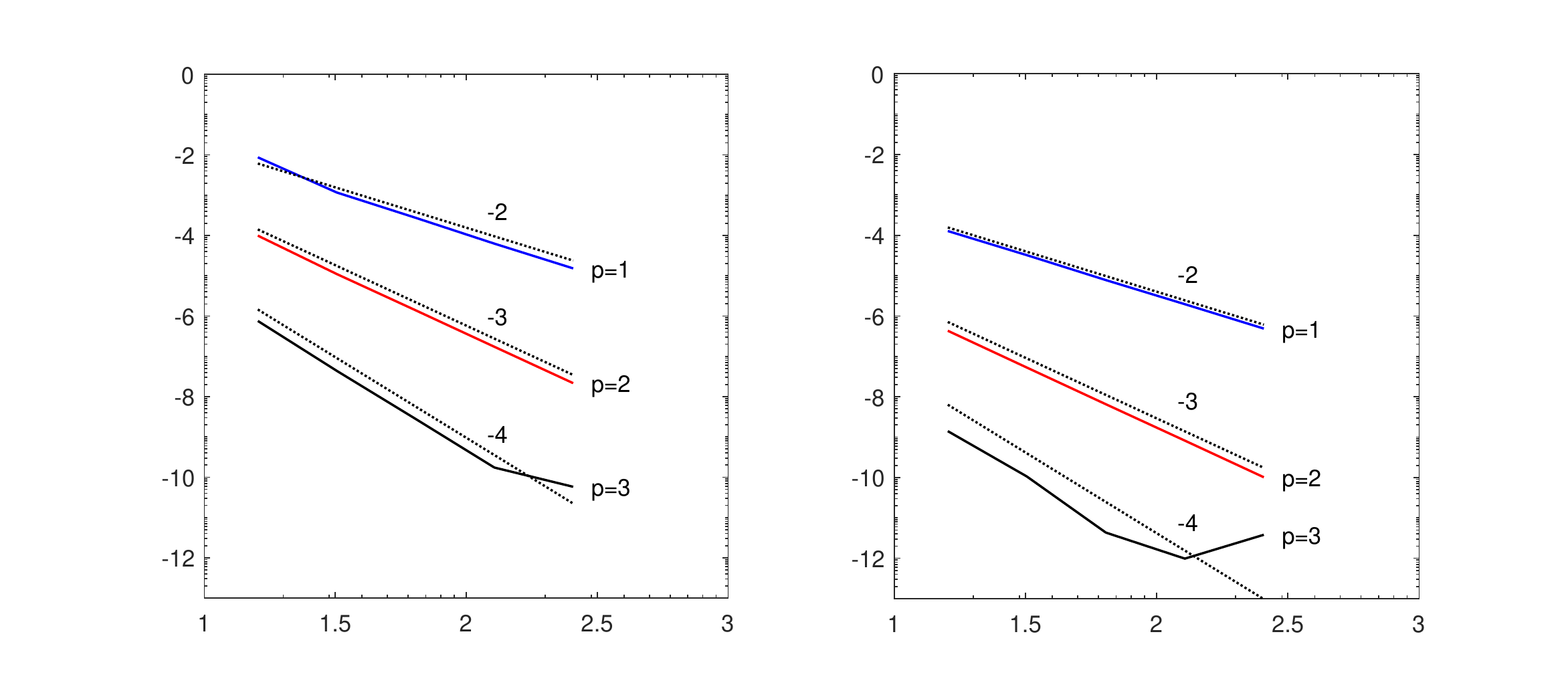}\label{F7}
				}
				\caption{$\lg\frac{\abs{u-u_h}_{1,h}}{\abs{u}_{1,h}}$ versus $\lg\frac{1}{h}$ in \eqref{F6} and $\lg\frac{\norm{u-u_h}_{0,\Om}}{\norm{u}_{0,\Om}}$ versus $\lg\frac{1}{h}$ in \eqref{F7}
				with $h=\frac{1}{16},\frac{1}{32},\cdots,\frac{1}{256}$ for $p=1,2,3,$ respectively. Left:
				$\al_1=1000,\al_2=1.$ Right: $\al_1=1,\al_2=1000.$ }
			\end{figure}
	
			In \cref{ex1}, the interface has a maximum curvature of approximately $\ka_m\approx50.3$. By Assumption \eqref{merging} with $\de=\frac14$, the theoretical threshold $h_0\le\frac{1}{58}$ guarantees successful execution of the merging algorithm. Practically, for the interface in \cref{ex1}, we observe that the algorithm remains robust even for $h_0=\frac{1}{16}$. We now present numerical results for the SUIPDG-FEM. Following \cite{hwx17}, the penalty parameter is set to $\ga=100$.

			First, we investigate the $h$-convergence rate of the SUIPDG-FEM. Denote the energy norm by
			\[\quad\abs{v}_{1,h}:=\big\|\al^{\frac{1}{2}}\nabla v\big\|_{0,\oo}.\]
			\cref{F6} presents the relative energy norm errors for both coefficient choices $\al(x)$ and polynomial degrees $p=1,2,3$. In all cases, the observed $O(h^p)$ convergence rate verify the theoretical estimate in \cref{Thm-Err-estimates-H1}. Similarly, \cref{F7} displays the relative $L^2$-norm errors, which exhibit the expected $O(h^{p+1})$ convergence order (\cref{Thm-Err-estimates-L2}), except for minor deviations in the final refinement steps when $p = 3$. These deviations are likely attributable to round-off errors, as the relative $L^2$-error falls below $10^{-10}$ at these stages.
			\cref{F8} demonstrates the $h$-convergence rate for the relative flux error, confirming the $O(h^p)$ theoretical bound in Theorem~\ref{Thm-Err-estimates-flux}.

			\begin{figure}[htbp]
				\centering
				\subfloat[]{
					\includegraphics[scale=0.26]{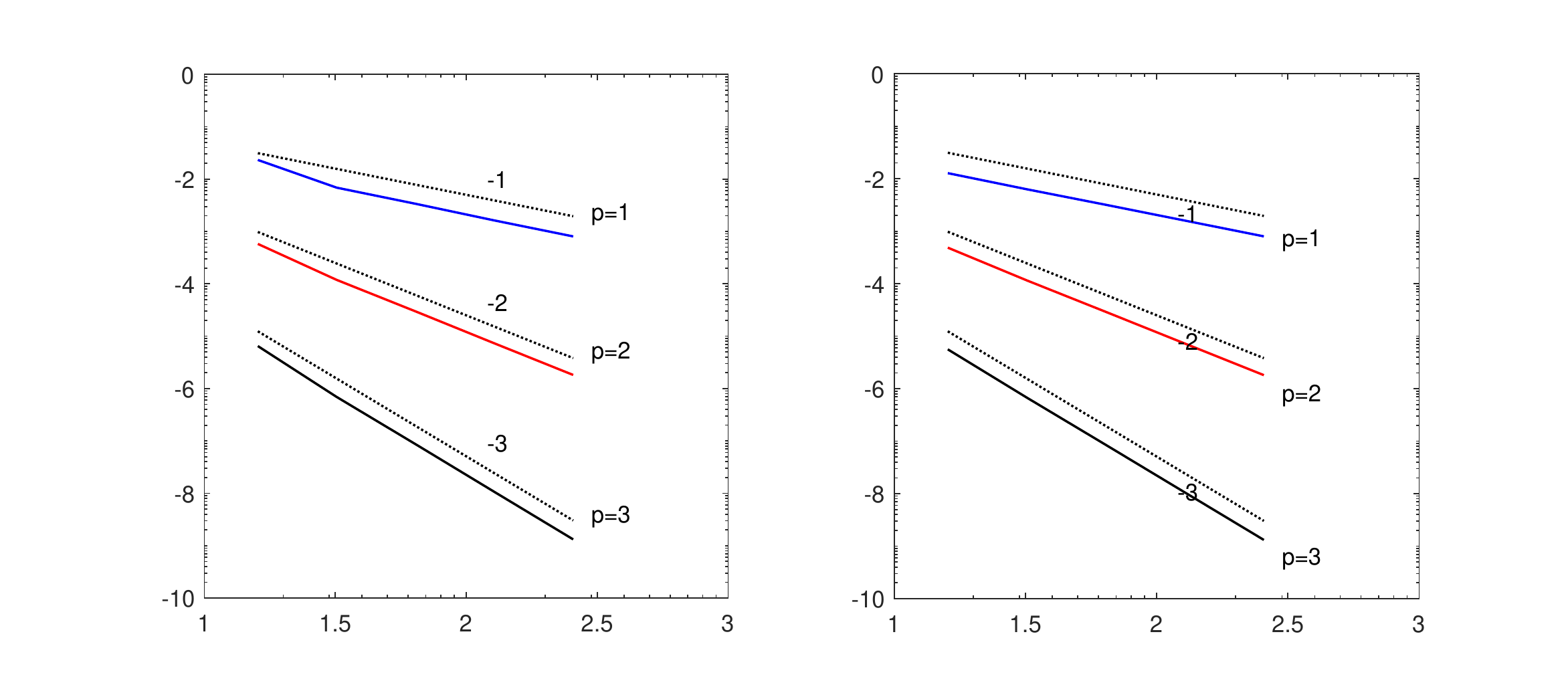}\label{F8}
				}
				\subfloat[]{
					\includegraphics[scale=0.26]{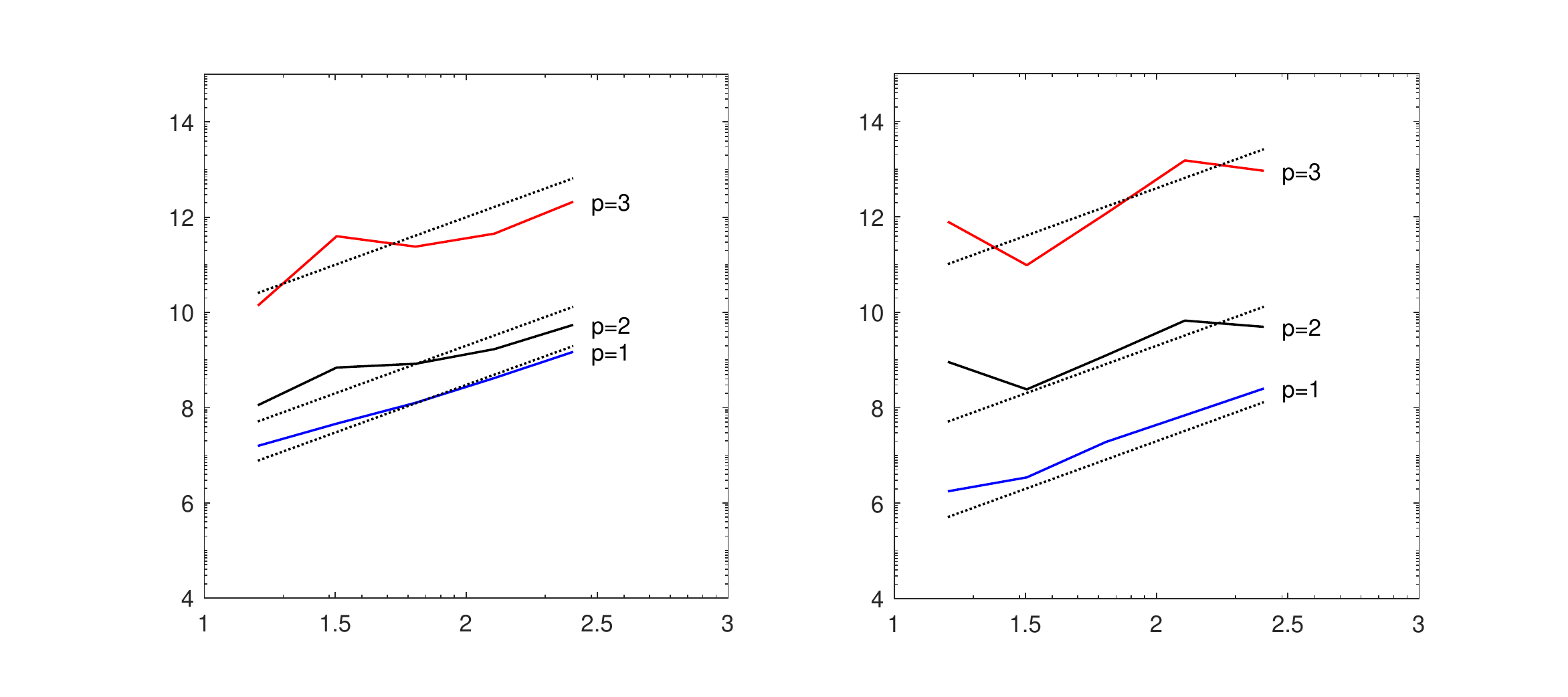}\label{F10}
				}
				\caption{$\lg\frac{\norm{\al\nabla(u-u_h)}_{0,\Om}}{\norm{\al\nabla u}_{0,\Om}}$ versus $\lg\frac{1}{h}$ in \eqref{F8} and $\lg{\rm cond}(A)$ versus $\lg\frac{1}{h}$ in \cref{F10}
				with $h=\frac{1}{16},\frac{1}{32},\cdots,\frac{1}{256}$ for $p=1,2,3,$ respectively. Left:
				$\al_1=1000,\al_2=1.$ Right: $\al_1=1,\al_2=1000.$ The dotted lines give reference lines of slopes $2$ in \eqref{F10}. }
			\end{figure}
			

			Next, we analyze the condition number $\mathrm{cond}(A)$ of the SUIPDG-FEM stiffness matrix. As shown in \cref{F10}, $\mathrm{cond}(A) = O(h^{-2})$, consistent with Theorem~\ref{Thm-condition-number}, since $\alpha_{\mathrm{max}}/\alpha_{\mathrm{min}} = 1000$ is fixed.
	
	
			Finally, we investigate the impact of the coefficient ratio $\alpha_2/\alpha_1$ on errors and conditioning. For a fixed mesh size $h = \frac{1}{32}$, \cref{F11} (a)-(c) plot $\lg\frac{\abs{u-u_h}_{1,h}}{\abs{u}_{1,h}}$, $\lg\frac{\norm{u-u_h}_{0,\Om}}{\norm{u}_{0,\Om}}$ and $\lg\frac{\norm{\al\nabla(u-u_h)}_{0,\Om}}{\norm{\al\nabla u}_{0,\Om}}$ versus $\lg\frac{\al_2}{\al_1}$, respectively. These results indicate that the relative $H^1$, $L^2$ and flux errors remain independent of $\alpha_{\mathrm{max}}/\alpha_{\mathrm{min}}$ as the ratio grows. Notably, the flux error (\cref{F11}(c)) is nearly constant with respect to $\alpha_2/\alpha_1$, fully consistent with \cref{rem-coefficient}. \cref{F11} (d) further confirms that $\mathrm{cond}(A)$ scales linearly with $\alpha_{\mathrm{max}}/\alpha_{\mathrm{min}}$, in agreement with Theorem~\ref{Thm-condition-number-0}.
	
			\begin{figure}[!htbp]
				\centering
				\includegraphics[scale=0.4]{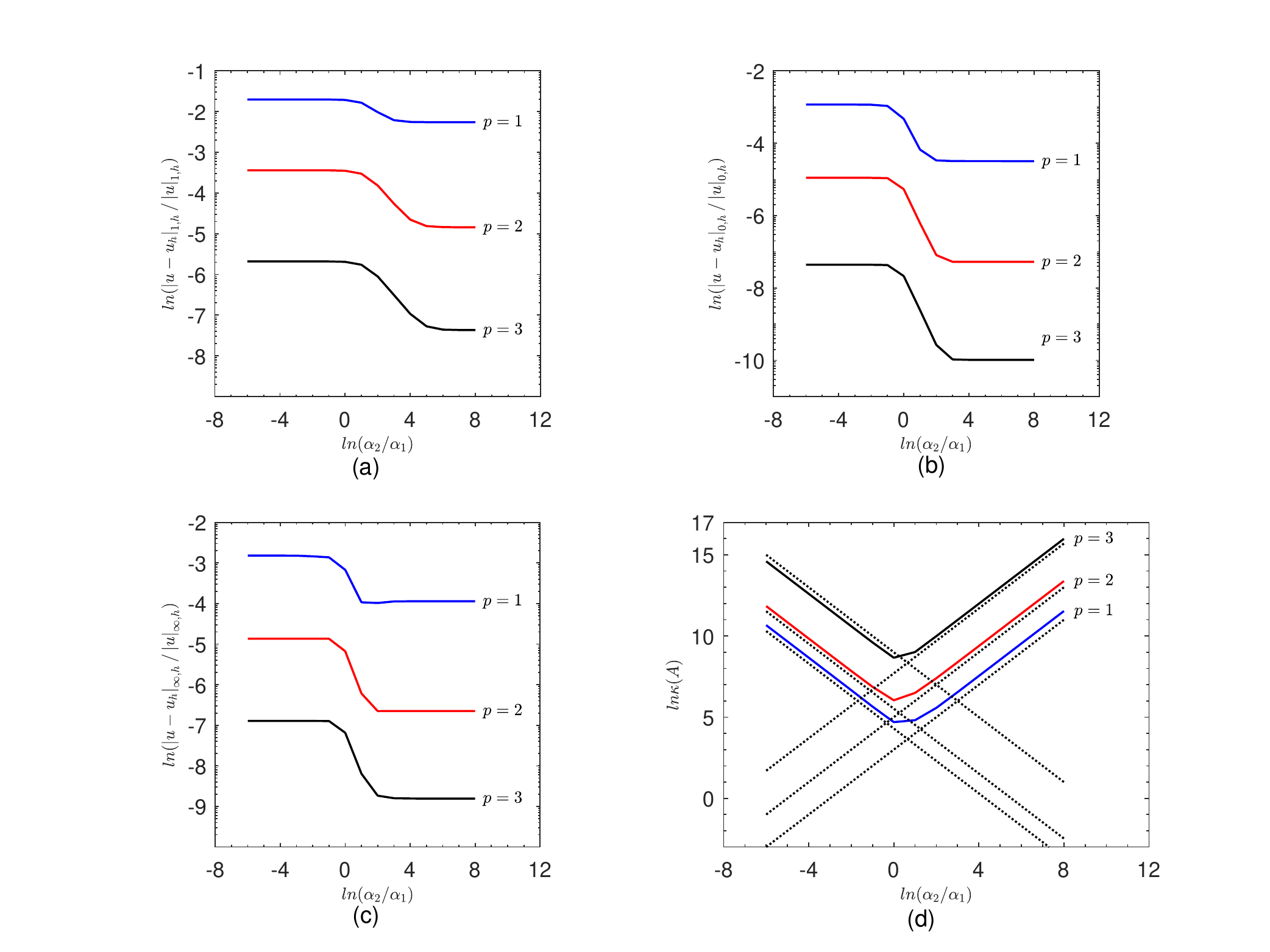}
				\caption{$\lg$-$\lg$ plots of the relative errors in $H^1$, $L^2$, flux norms, and
				${\rm cond}(A)$ versus $\al_2/\al_1$ with $(\al_1,\al_2)=(10^6,1),(10^5,1),\cdots,(1,1),(1,10),\cdots,(1,10^8),$
				for $p=1,2,3$, respectively. The dotted lines in (d) give reference lines of slopes $-1$ or $1$.}\label{F11}
			\end{figure}

		\bibliographystyle{abbrv} 
	\bibliography{ref}

\begin{thebibliography}{10}

\bibitem{arnold82}
D.~Arnold.
\newblock An interior penalty finite element method with discontinuous elements.
\newblock {\em SIAM J. Numer. Anal.}, 19:742--760, 1982.

\bibitem{Babuvska72}
I.~Babu\v{s}ka.
\newblock The finite element method with {L}agrangian multipliers.
\newblock {\em Numer. Math.}, 20:179--192, 1972/73.

\bibitem{SFA18}
S.~Badia, F.~Verdugo, and A.~F. Mart\'{\i}n.
\newblock The aggregated unfitted finite element method for elliptic problems.
\newblock {\em Comput. Methods Appl. Mech. Engrg.}, 336:533--553, 2018.

\bibitem{bastian2009unfitted}
P.~Bastian and C.~Engwer.
\newblock An unfitted finite element method using discontinuous {G}alerkin.
\newblock {\em Int. J. Numer. Meth. Engng}, 79(12):1557--1576, 2009.

\bibitem{Bu10}
E.~Burman.
\newblock Ghost penalty.
\newblock {\em Comptes Rendus Mathematique}, 348:1217--1220, 2010.

\bibitem{BCDE21}
E.~Burman, M.~Cicuttin, G.~Delay, and A.~Ern.
\newblock An unfitted hybrid high-order method with cell agglomeration for elliptic interface problems.
\newblock {\em SIAM J. Sci. Comput.}, 43(2):A859--A882, 2021.

\bibitem{BE07}
E.~Burman and A.~Ern.
\newblock Continuous interior penalty $hp$-finite element methods for advection and advection-diffusion equations.
\newblock {\em Math. Comp.}, 259:1119--1140, 2007.

\bibitem{bgss16}
E.~Burman, J.~Guzm\'{a}n, M.~A. S\'{a}nchez, and M.~Sarkis.
\newblock Robust flux error estimation of an unfitted {N}itsche method for high-contrast interface problems.
\newblock {\em IMA J. Numer. Anal.}, 38(2):646--668, 2018.

\bibitem{bhl17}
E.~Burman, P.~Hansbo, and M.~G. Larson.
\newblock A cut finite element method with boundary value correction.
\newblock {\em Math. Comp.}, 87(310):633--657, 2018.

\bibitem{CYZ11}
Z.~Q. Cai, X.~Ye, and S.~Zhang.
\newblock Discontinuous {G}alerkin finite element methods for interface problems: a priori and a posteriori error estimations.
\newblock {\em SIAM J. Numer. Anal.}, 49:1761--1787, 2011.

\bibitem{chen2022}
Z.~Chen and Y.~Liu.
\newblock An arbitrarily high order unfitted finite element method for elliptic interface problems with automatic mesh generation.
\newblock {\em J. Comput. Phys.}, 491:Paper No. 112384, 24, 2023.

\bibitem{chen2015adaptive}
Z.~Chen, Z.~Wu, and Y.~Xiao.
\newblock An adaptive immersed finite element method with arbitrary lagrangian-eulerian scheme for parabolic equations in time variable domains.
\newblock {\em International Journal of Numerical Analysis \& Modeling}, 12(3), 2015.

\bibitem{CZ98}
Z.~Chen and J.~Zou.
\newblock Finite element methods and their convergence for elliptic and parabolic interface problems.
\newblock {\em Numer. Math.}, 79:175--202, 1998.

\bibitem{cgh10}
C.~Chu, I.~G. Graham, and T.~Hou.
\newblock A new multiscale finite element method for high-contrast elliptic interface problems.
\newblock {\em Math. Comp.}, 79(272):1915--1955, 2010.

\bibitem{Ci78}
P.~G. Ciarlet.
\newblock {\em The Finite Element Method for Elliptic Problems}.
\newblock North-Holland, Amsterdam, 1978.

\bibitem{eg06}
A.~Ern and J.~Guermond.
\newblock Evaluation of the condition number in linear systems arising in finite element approximations.
\newblock {\em M2AN}, 40:29--48, 2006.

\bibitem{EG04}
A.~Ern and J.-L. Guermond.
\newblock {\em Theory and Practice of Finite Elements}.
\newblock Springer-Verlag, New York, 2004.

\bibitem{esz09}
A.~Ern, A.~F. Stephansen, and P.~Zunino.
\newblock A discontinuous {G}alerkin method with weighted averages for advection-diffusion equations with locally small and anisotropic diffusivity.
\newblock {\em IMA J. Numer. Anal.}, 29:235--256, 2009.

\bibitem{FAMO99}
R.~P. Fedkiw, T.~Aslam, B.~Merriman, and S.~Osher.
\newblock A non-oscillatory {E}ulerian approach to interfaces in multimaterial flows (the ghost fluid method).
\newblock {\em J. Comput. Phys.}, 152:457--492, 1999.

\bibitem{FW09}
X.~Feng and H.~Wu.
\newblock Discontinuous {G}alerkin methods for the {H}elmholtz equation with large wave numbers.
\newblock {\em SIAM J. Numer. Anal.}, 47:2872--2896, 2009.

\bibitem{fo06}
T.~Fries and S.~Omerovi\'c.
\newblock Higher-order accurate integration of implicit geometries.
\newblock {\em Int. J. Numer. Meth. Engng}, 106:323--371, 2016.

\bibitem{GR86}
V.~Girault and P.-A. Raviart.
\newblock {\em Finite element methods for {N}avier-{S}tokes equations}, volume~5 of {\em Springer Series in Computational Mathematics}.
\newblock Springer-Verlag, Berlin, 1986.
\newblock Theory and algorithms.

\bibitem{HH02}
A.~Hansbo and P.~Hansbo.
\newblock An unfitted finite element method, based on {N}itsche's method, for elliptic interface problems.
\newblock {\em Comput. Methods Appl. Mech. Engrg.}, 191:5537--5552, 2002.

\bibitem{hu2021optimal}
J.~Hu and H.~Wang.
\newblock An optimal multigrid algorithm for the combining p 1-q 1 finite element approximations of interface problems based on local anisotropic fitting meshes.
\newblock {\em Journal of Scientific Computing}, 88(1):16, 2021.

\bibitem{hz07}
J.~Huang and J.~Zou.
\newblock Uniform a priori estimates for elliptic and static {M}axwell interface problems.
\newblock {\em Discrete and Continuous Dynamical Systems, Series B}, 7(1):145--170, 2007.

\bibitem{hwx17}
P.~Huang, H.~Wu, and Y.~Xiao.
\newblock An unfitted interface penalty finite element method for elliptic interface problems.
\newblock {\em Comput. Methods Appl. Mech. Engrg.}, 323:439--460, 2017.

\bibitem{JL13}
A.~Johansson and M.~G. Larson.
\newblock A high order discontinuous {G}alerkin {N}itsche method for elliptic problems with fictitious boundary.
\newblock {\em Numer. Math.}, 123:607--628, 2013.

\bibitem{L16}
C.~Lehrenfeld.
\newblock High order unfitted finite element methods on level set domains using isoparametric mappings.
\newblock {\em Comput. Methods Appl. Mech. Engrg.}, 300:716--733, 2016.

\bibitem{ll94}
R.~LeVeque and Z.~L. Li.
\newblock The immersed interface method for elliptic equations with discontinuous coefficients and singular sources.
\newblock {\em SIAM J. Numer. Anal.}, 31:1019--1044, 1994.

\bibitem{lmwz10}
J.~Li, J.~M. Melenk, B.~Wohlmuth, and J.~Zou.
\newblock Optimal a priori estimates for higher order finite elements for elliptic interface problems.
\newblock {\em Appl. Numer. Math.}, 60:19--37, 2010.

\bibitem{LY20}
R.~Li and F.~Yang.
\newblock A discontinuous {G}alerkin method by patch reconstruction for elliptic interface problem on unfitted mesh.
\newblock {\em SIAM J. Sci. Comput.}, 42(2):A1428--A1457, 2020.

\bibitem{LLW03}
Z.~Li, T.~Lin, and X.~Wu.
\newblock New {C}artesian grid methods for interface problems using the finite element formulation.
\newblock {\em Numer. Math.}, 96:61--98, 2003.

\bibitem{LLZ15}
T.~Lin, Y.~Lin, and X.~Zhang.
\newblock Partially penalized immersed finite element methods for elliptic interface problems.
\newblock {\em SIAM J. Numer. Anal.}, 53(2):1121--1144, 2015.

\bibitem{m09}
R.~Massjung.
\newblock An unfitted discontinuous {G}alerkin method applied to elliptic interface problems.
\newblock {\em SIAM J. Numer. Anal.}, 50:3134--3162, 2012.

\bibitem{Nitsche}
J.~A. Nitsche.
\newblock \"uber ein variationsprinzip zur l\"osung {D}irichlet-{P}roblemen bei {V}erwendung von {T}eilr\"aumen, die keinen {R}andbedingungen unteworfen sind.
\newblock {\em Abh. Math. Sem. Univ. Hamburg}, 36:9--15, 1971.

\bibitem{Pe77}
C.~S. Peskin.
\newblock Numerical analysis of blood flow in heart.
\newblock {\em J. Comput. Phys.}, 25:220--252, 1977.

\bibitem{S15}
R.~I. Saye.
\newblock High-order quadrature methods for implicitly defined surfaces and volumes in hyperrectangles.
\newblock {\em SIAM J. Sci. Comput.}, 37:A993--A1019, 2015.

\bibitem{WX10}
H.~Wu and Y.~Xiao.
\newblock An unfitted hp-interface penalty finite element method for elliptic interface problems.
\newblock {\em J. Comput. Math.}, 37:316--339, 2019.

\bibitem{zy2024}
H.~Zhou and W.~Ying.
\newblock A correction function-based kernel-free boundary integral method for elliptic {PDE}s with implicitly defined interfaces.
\newblock {\em J. Comput. Phys.}, 496:Paper No. 112545, 18, 2024.

\end{thebibliography}
\end{document}